\documentclass[a4paper,twoside,english]{amsart}
\usepackage[T1]{fontenc}
\usepackage[latin9]{inputenc}
\usepackage{babel}
\usepackage{verbatim}
\usepackage{amsthm}
\usepackage{amsmath}
\usepackage{amssymb}
\usepackage[unicode=true,
 bookmarks=true,bookmarksnumbered=true,bookmarksopen=false,
 breaklinks=false,pdfborder={0 0 1},backref=false,colorlinks=false]
 {hyperref}
\hypersetup{
 pdfauthor={V. Mantova and U. Zannier}}

\makeatletter

\pdfpageheight\paperheight 
\pdfpagewidth\paperwidth

\theoremstyle{plain}
\newtheorem{thm}{\protect\theoremname}[section]
 \newcommand\thmsname{\protect\theoremname}
 \newcommand\nm@thmtype{theorem}
 \theoremstyle{plain}

  \theoremstyle{remark}
  \newtheorem{rem}[thm]{\protect\remarkname}
  \theoremstyle{definition}
  \newtheorem*{example*}{\protect\examplename}
  \theoremstyle{definition}
  \newtheorem{example}[thm]{\protect\examplename}
  \theoremstyle{plain}
  \newtheorem{lem}[thm]{\protect\lemmaname}
  \theoremstyle{plain}
  \newtheorem{prop}[thm]{\protect\propositionname}
  \theoremstyle{plain}
  \newtheorem{cor}[thm]{\protect\corollaryname}

\makeatletter
  \theoremstyle{definition}
  \newtheorem{my@rem}[thm]{Remark}
  \renewenvironment{rem}{\begin{my@rem}}{\end{my@rem}}
\makeatother

\makeatother

  \providecommand{\examplename}{Example}
  \providecommand{\lemmaname}{Lemma}
  \providecommand{\propositionname}{Proposition}
  \providecommand{\remarkname}{Remark}
  \providecommand{\theoremname}{Theorem}
\providecommand{\theoremname}{Theorem}
 \providecommand{\corollaryname}{Corollary}

\usepackage[left=3.5cm,right=3.3cm,top=4cm,bottom=3cm]{geometry}

\usepackage{mathtools}

\def\N{{\Bbb N}}
\def\Q{{\Bbb Q}}

\def\Z{{\Bbb Z}}
\def\G{{\Bbb G}}

\def\O{{\mathcal O}}

\def\Gal{\mathop{\rm Gal}\nolimits}

\def\div{{\rm div}}

\def\P{{\Bbb P}}
\def\C{{\Bbb C}}
\def\F{{\Bbb F}}

\def\d{{\rm d}}

\def\X{{\mathcal X}}

\def\CVD{{\hfill\hfil{\lower 2pt\hbox{\vrule\vbox to 7pt
{\hrule width 6pt\vfill\hrule}\vrule}}}\par}

\begin{document}

\title{Congruences modulo $p$ for a certain sequence and questions which arise}
\author{U.\ Zannier}

\maketitle

\begin{abstract}

We shall investigate congruence properties of a certain specific sequence of rational numbers, satisfying a linear recurrence with (linear) polynomial coefficients. Although such example is quite special and not especially significant among similar ones, it will serve as an instance leading us 
 to touch a number of mathematical topics of various flavour, and  linked among them: questions concerning linear differential equations over (elliptic) curves, integrality properties of the Taylor coefficients, algebraicity of the solutions, congruences modulo $p$ for differential forms and even special cases of a well-known conjecture of Grothendieck. 

Results and method of this article can be considered to be original  at most in 
part, so the paper is  mainly of expository nature. It shall clearly appear, without the need to make it explicit,  that the contents may be 
generalized; so  we have preferred to stick to our special sequence for the sake of simplicity, in order to keep symbols and  details to a minimum.  Our methods will be essentially elementary, except that for some conclusions we shall rely on a deep theorem by Yves Andr\'e;  anyway we shall give suitable references for all the background material we shall use.

\end{abstract}

\bigskip

\section{Introduction} Sequences defined by (linear) recurrences have been very often a source of inspiration for several  mathematicians; 
 of course the best known prototypes are the Fibonacci sequence, the  exponential sequences $(a^n)_{n\in\N}$ and in general the solutions of  linear recurrences with constant coefficients.  Clearly it is natural to consider as well non-constant coefficients; for instance, polynomial coefficients appear in the study of linear differential equations over $\C(x)$, in particular satisfied by algebraic functions. The arithmetical study here  presents fundamental new issues and originated a number of difficult results and conjectures, still in great part out of reach. 

In the present paper we shall discuss in a rather elementary way some of these issues using a kind of `test example', 
a quite specific 
sequence of rational numbers, satisfying a linear recurrence whose coefficients are linear polynomials.  This sequence was met by somewhat  `random' reasons, and  we   
 happened to prove in a simple way some congruence properties, rather  unexpected to us. 
 We then realized that the matter was in rather close link with certain topics of general mathematical interest, which led to an expansion of  our analysis and to some results and questions which we could not locate in the literature.  
 
 After gathering together all of this, we thought that the special example originating this small research,  
 even if not particularly significant among similar ones, could be amusing for some readers and might serve as an introduction to some of the said topics, addressed to non-specialists (like the writer). This gave rise to the present article. With the said purpose in mind, we  decided to limit ourselves to the original example, and not to discuss straightforward generalizations, thinking that this would only make the whole more complicated, without introducing novel points; the interested reader shall easily see that many similar and more involved examples can be constructed along the same lines.
 
 \medskip
 
 We shall now state in precision the original context and the results, and then comment on them, pointing out in some detail the said links. The next sections shall be mainly devoted to proofs, except for the last section, devoted to comments and   related question.  
 
   To start with, here is the recurrence in question:

\begin{equation}\label{E.recurrence}
4(n+4)\gamma_{n+5}+8(n+2)\gamma_{n+4}+(n+3)\gamma_{n+3}+(4n+7)\gamma_{n+2}+(5n+4)\gamma_{n+1}+2n\gamma_n=0,
\end{equation}
for $n\ge 0$.

If we work over a field $\kappa$ of characteristic zero, this recurrence determines uniquely a sequence $(\gamma_n)_{n\in\N}$ as soon as we prescribe initial values  $\gamma_0,\ldots ,\gamma_4\in\kappa$. \footnote{The case of positive characteristic is different and shall occupy substantial space in what follows.} We also note at once that the value $\gamma_0$ is immaterial, since it does not affect the subsequent ones; hence we shall often suppose $\gamma_0=0$.

 We now concentrate on the sequence $(c_n)_{n\in\N}$ of rational numbers obtained by choosing:
\begin{equation}\label{E.initial}
c_0=0,\ c_1=1,\ c_2=2,\ c_3=-{1\over 8},\ c_4=-{1\over 2},
\end{equation}
For instance, we find $c_5=-77/128$. It may be worth observing that these values are, up to a constant factor and up to changing $c_0$, the unique values such that the extended sequence defined by $c_m=0$ for $m<0$ continues to satisfy \eqref{E.recurrence}.

The above recurrence is linear but is not of the more common type with constant coefficients (whose prototype corresponds to  the Fibonacci numbers). In that case the solutions are well-known to be given by exponential polynomials, and plainly the terms carry at most finitely many primes in the denominators. 
On the contrary here  in obtaining $c_n$  from the previous terms  we perform a division by  $4(n-1)$, hence {\it a priori} one could guess that the denominators would grow like factorials. (This indeed happens for other 
 linear recurrences, like $(n+1)c_{n+1}=c_n$, again having coefficients which are polynomials of degree $\le 1$.\footnote{Another simple, but less obvious,  example  comes from the recurrence $(n+2)c_{n+2}-(2n+3)c_{n+1}+nc_n=0$, with $c_0=c_1=1$; now  {\it all} the polynomial coefficients have degree $1$ and   again the whole sequence $(c_n)_{n\in \Z}$ defined by $c_m=0$ for $m<0$ continues to satisfy the recurrence. The generating function is $\exp(x/(1-x))$, and on substituting $x/(1-x)=y$ it may be proved that the common denominator of $c_0,\ldots ,c_n$ grows like a factorial.}) Nevertheless, it turns out that  the $c_n$ so obtained are $p$-integers for every odd prime and, more surprisingly, they even satisfy nontrivial congruences modulo each such prime. 

We resume these facts in precision with the following

\medskip

\begin{thm}\label{T.congr}  The $c_n$ do not satisfy any  linear recurrence with constant coefficients.  They are rational numbers whose denominator is a power of $2$.  

Moreover, for any odd prime $p$, they satisfy the congruences 
\begin{equation}\label{E.congr}
c_{kp^{r+1}}\equiv c_{kp^r}\pmod {p^{r+1}},\qquad  k,r=0,1,\ldots .
\end{equation}
\end{thm}

\medskip

As to the integrality assertion, the proof shall easily yield that $2^{2n-3}c_n$ is an integer; also, for infinitely many indices $n$,  the exponent $2n-3$ cannot be  replaced by a smaller number (for this more precise conclusion, see  \S \ref{SS.explicit}). 

As to the congruences, note for instance the special case  $c_p\equiv c_1\pmod p$, already somewhat surprising (at least for us). Also,  this especially reminds of Fermat's   Little Theorem, which indeed  implies   the same congruences (\ref{E.congr})   if  $c_n$ is replaced for instance by an exponential sum  
$e_n:=\sum_{i=1}^ha_ib_i^n$, where $a_i,b_i$ are given integers;   the  $e_n$ 
then satisfy a linear recurrence with constant coefficients. However, the first statement in the theorem implies that  the present sequence is not of this especially simple type. 
In such exponential cases, it turns out that actually stronger congruences hold, and for instance  the sequence $(e_n)$  is  eventually periodic modulo $p$ for every prime.  We shall note that, on the contrary, this periodicity does not occur for the $c_n$. Nevertheless, other    congruences hold, as for instance $6c_{kp+4}+c_{kp+2}+c_{kp+1}\equiv 0\pmod p$ for all $k\in\N$, or also $c_{kp-2}+c_{kp-3}\equiv 0\pmod p$. These and other similar ones can be derived very simply as in Example \ref{E.modp} below.
As   we shall point out in \S \ref{SS.cartier}, they do not hold in stronger form, as e.g., $c_{kp+i}\equiv c_i\pmod p$ for $i=1,2,4$.

\medskip 

It seems natural to ask what happens with  other initial data; let us study  this issue. Note that if we 
start with $a, b, 2b, -b/8, -b/2$, where $a,b\in\Q$, the resulting sequence will have $bc_n$ in place of $c_n$ for $n\ge 1$, so we shall continue to have rational numbers with denominator   dividing a constant times a power of $2$, and the congruences shall continue to hold provided $p$ does not divide the denominator of $b$.





On the other hand, on trying with still other initial values the resulting sequence seems  to carry increasing primes in the denominators, as should be expected. 
 We shall    confirm this expectation,  by proving (in particular) that   for   initial data 
different from the said ones the resulting denominators carry infinitely many primes (actually a set of positive upper asymptotic density). In this we shall use a rather deep theorem in the context of $G$-functions.


However, we shall easily see that, no matter the (rational) initial data, the denominators increase at most exponentially, rather than  factorials.  

We shall give actually more precise conclusions, which can be regarded as complements to  the above theorem.


We have the following results, where $(C_n)_{n\in\N}$ denotes a sequence of rational numbers satisfying (\ref{E.recurrence}). 


\medskip

\begin{thm} \label{T.converse}  (i) 
There exists an integer $d\neq 0$ such that the  common denominator  of $d4^{2m}C_m$  for $m\le n$, divides  the least common multiple of the integers up to $n-1$. 

(ii)  If  the vector $(C_1,C_2,C_3,C_4)$ is not 
proportional to $(1,2,-1/8,-1/2)$, the set of primes $p$ such that the sequence $(C_n)$ is $p$-integral has not density $1$ (i.e.,  the complement has positive upper asymptotic density).

\end{thm}

By `positive upper asymptotic density' we mean that there are a positive $c$ and   intervals $[1,T]$ with arbitrarily large $T$ which contain $>c T/\log T$ relevant primes.

\medskip

As we shall see, the proof of part (ii) relies on a certain deep theorem of Y. Andr\'e; in personal communication, he pointed out to us that  his proof in fact leads to the conclusion that the  upper density of the  relevant set of primes is $\le 1/2$. Probably,  an even stronger conclusion 
holds, perhaps even of finiteness,   or at any rate `half-way', with the   set of primes such that the $C_n$ are $p$-integral having density $0$. However we do not know how to prove any of these more precise statements.  See below for some results derived from  an analysis modulo $p$.
Also, in the last section we shall   discuss this issue in a little more depth. 
We shall also point out effectivity issues; for instance, for given rational initial data $C_0,\ldots ,C_4$, it does not seem easy to decide effectively and in general whether the $C_n$ so obtained are $p$-adically integral. 

As to conclusion (i), recall that the Prime Number  Theorem  (see \cite{[I]}) implies that the least common multiple of $1,2,\ldots ,n$ grows like $\exp(n+o(n))$, which clearly implies in particular the above alluded  exponential upper bound for the denominators of the $C_n$.  Thinking of quantitative lower bounds, the proof-argument for part (ii) of the theorem shall be also shown to  imply  the following  estimate, in the opposite direction compared to  (i). 

\begin{prop}\label{C.lowerbound}
 If  the vector $(C_1,C_2,C_3,C_4)$ is not proportional to $(1,2,-1/8,-1/2)$, there exist a number $a>1$ and infinitely many $n$ such that the product of primes dividing  the least common denominator of $C_1,\ldots ,C_n$ is $> a^n$. 
 
\end{prop}

Note that   the first assertion  represents indeed a sharpening of the  second conclusion in the previous theorem; in fact, by elementary prime number estimates (for which we again refer to \cite{[I]}),  the logarithm of the product of primes up to $T$ in a set of zero density is $o(T)$. 

\bigskip




\noindent\underline{\sc The case modulo $p$.} The  results expressed by these last two statements concern integrality properties of solutions of the recurrence  \eqref{E.recurrence} in  rational numbers. But we can also study what happens for solutions  over $\F_p$;  this may be of interest in itself and, through reduction, might also give further  informations  on integrality in characteristic zero. 

 Let $(\bar C_n)_{n\in\N}$ be a sequence over $\F_p$, satisfying \eqref{E.recurrence}; these sequences make up a vector space over $\F_p$, denoted $W_p$ in what follows. \footnote{ Generally, in the paper we shall denote with a bar coefficients in $\F_p$, even if they are not necessarily the reductions of corresponding rational numbers satisfying the recurrence.}

  In this  case of positive characteristic   the situation is different in (at least) two points: first of all,    the initial data might not determine the solution uniquely, since we have a free choice for $\bar C_n$ each time it happens that $p|n-1$; indeed, we shall easily see that $W_p$  is infinite-dimensional over $\F_p$.  But it is worth noticing that the values $\bar C_1,\ldots ,\bar C_4)$ determine uniquely the $\bar C_n$ for $n=1,\ldots ,p$.

  Secondly, if $p|n-1$, the values $\bar C_5,\ldots ,\bar C_{n-1}$ obtained so far (starting from the initial values $\bar C_0,\ldots ,\bar C_4$ and performing the said `free choices') might turn out to be incompatible, since they have to satisfy $16\bar C_{n-1}+\bar C_{n-2}+9\bar C_{n-3}+16\bar C_{n-4}+8\bar C_{n-5}=0$; hence some  initial data {\it a priori} may not determine a solution at all.  We observe again that they determine the entries $\bar C_5,\ldots ,\bar C_p$.

For sequences modulo $p$, these observations are confirmed, in particular, by  the following result:

\begin{thm}\label{T.modp} The vector space $W_p$ consisting of the   sequences $(\bar C_n)_{n\in\N}$   over $\F_p$ satisfying \eqref{E.recurrence} has infinite dimension.  For  $p\neq 2, 3, 5, 13$,   its image $V_p$ under the projection $(\bar C_n)\to (\bar C_1,\ldots ,\bar C_4)\in\F_p^4$ has dimension $2$ and is contained in the hyperplace of $\F_p^4$ defined by  $6\bar C_4+\bar C_2+\bar C_1=0$. Also, $V_p$  is generated by $(1,2,-1/8,-1/2)$ and  $(\bar c_{p+1},\bar c_{p+2},\bar c_{p+3}, \bar c_{p+4})$ (i.e. the reduction of $(c_{p+1},\ldots ,c_{p+4})$).  \footnote{ The exclusion of $2,5,13$ shall become clear in the sequel. For these primes the result is not true and the condition becomes different; for instance, for $p=5$ it is that $\bar C_3-\bar C_2-\bar C_1=0$ in $\F_5$. We leave it to the interested reader to carry out the analysis in these cases.} 

\end{thm}

 In particular, we see that, at least for primes $\neq 2,3, 5,13$, there are always initial data $\bar C_1,\ldots ,\bar C_4$ not proportional to $(1,2,-1/8,-1/2)$ which may be extended to a sequence  over $\F_p$   satisfying \eqref{E.recurrence}. 
 
 \medskip

This  result  yields in particular the following corollary, partly in the direction of Theorem \ref{T.converse} and Proposition \ref{C.lowerbound}:

\begin{cor}\label{C.converse} Let $(C_n)_{n\in\N}$ be a sequence of rational numbers satisfying (\ref{E.recurrence}). If $6C_4+C_2+C_1\neq 0$, there are only finitely many primes such that the sequence consists of $p$-integers. Also, the estimate of Proposition \ref{C.lowerbound}  holds for all sufficiently large $n$.

\end{cor}


\medskip
 
 Concerning a proof of  a part of Theorem \ref{T.modp} we shall point out more than one approach, giving a further bit of information on  the structure of $V_p$.  For instance, using a deep result of Elkies  we shall observe that, given any   $\Z$-linear form independent of $6\bar C_4+\bar C_2+\bar C_1$, there are infinitely many primes such that this form does  not vanish on $V_p$; equivalently, for every vector subspace of $\Q^4$ of dimension $2$, there are infinitely many primes such that $V_p$ is not the reduction of this subspace. This yields in particular a sharpening of the Corollary and a completely different proof of a weak version of Theorem \ref{T.converse}(ii), in which `...has not density $1$' is replaced by `...has not finite complement'. Unfortunately  we cannot say much more  on 
 this set of primes; the issue seems relevant for a general question on Elliptic curves over finite fields, which we shall mention.

Also, contrary to the mentioned case of $p$-integrality in characteristic zero, one may effectively decide whether given intial data in $\F_p$   extend to a  whole sequence over  in $V_p$, i.e. whether they correspond to a(t least one)  sequence satisfying \eqref{E.recurrence}. See Remark \ref{R.estimate2}, and the last section, for further  comment and details on these issues.

\medskip


Now, in view of the congruences \eqref{E.congr}  satisfied by the $c_n$, it makes sense to ask whether, for other initial data 
and for the possible primes such that  all the $C_n$ are $p$-integral, such congruences continue to hold. Here we may prove that, as expected, the initial data producing  the $c_n$ are quite peculiar in this sense, even limiting to the case $r=0, k=1$ of \eqref{E.congr}; indeed, we have:




 \begin{thm}\label{T.union}   Let, for a prime $p\neq 2,3,5,13$,    a sequence $(\bar C_n)$, with $\bar C_n\in\F_p$, satisfy  the recurrence \eqref{E.recurrence}. 
 Then $\bar C_p=\bar C_1$ if and only if  
 $(\bar C_1,\bar C_2,\bar C_3,\bar C_4)$ is proportional to the reduction of $(1,2,-1/8,-1/2)$ modulo $p$.
 \end{thm}
 
 
 Correspondingly, we have:

\begin{cor}\label{C.union} Let $(C_n)_{n\in\N}$ be a sequence of rational numbers satisfying (\ref{E.recurrence}) and suppose that   $(C_1,C_2,C_3,C_4)$ is not proportional to $(1,2,-1/8,-1/2)$. Then there are only finitely many primes $p$ such that all the $C_n$ are $p$-integral and satisfy $C_{p}\equiv C_1\pmod p$. 
\end{cor}

\medskip


\subsection{Comments on the results.}\label{SS.comments} Let us now briefly comment on the nature and proofs of these results.

\medskip

\underline{\it Theorem \ref{T.congr}} : The first part of Theorem \ref{T.congr} shall follow immediately on realizing the generating function $s(x)$ of the sequence $\{c_n\}$ as a solution of a suitable first order linear differential equation over $\Q(x)$, and in turn as a certain algebraic function (of degree $2$ over $\Q(x)$), indeed a rational function on the curve of genus $1$ defined by $y^2=4+(x+x^2)^2$. It shall appear that $s(x)$ is not in $\Q(x)$, so the $c_n$ do not satisfy any  linear recurrence with constant coefficients.

The second part is a bit more hidden, and depends on the fact that 
$x^{-1}s(x)$ is a logarithmic derivative of another algebraic function (in the same function field). 
Somewhat surprisingly, all power series of such a shape (forgetting their algebraicity) give rise to {\it the same} congruence properties; for instance, we shall find that Fermat's Little Theorem  follows as a special case. To develop all of  this shall be the object of the short Section \ref{S.congr}, where we shall provide simple self-contained proofs (together with a number of references to other known proofs). 

We shall also 
recall a (known) {\it converse}, i.e. that, for a given prime $p$, such    congruence properties of a sequence of integers may hold only if  the   generating function is $x$ times a logarithmic derivative of a power series with $p$-integral coefficients. So also the congruence properties amount  to an  integrality condition (for the coefficients of a related power series).

All of these facts  
  represent   very simple instances of a well-known and widely investigated conjecture of A. Grothendieck, which (in the simplest case of the rational field) roughly speaking predicts that {\it if a linear differential equation over $\Q(x)$ has a full set of  solutions modulo $p$ for all large primes $p$, then it has a full set of algebraic solutions}.\footnote{ Here and in the sequel, if algebraic non rational quantities are involved, by `modulo $p$' we shall mean   reduction modulo some place of $\overline\Q$ above $p$.  By {\it full set} we mean a maximal set - that is, of cardinality equal to the order of the differential equation - of linearly independent solutions over the constant field;  in the modulo $p$ case, this field is $\F_p(x^p)$. There are several equivalent formulations of this conjecture of Grothendieck; see \cite{[A2]}, also for a generalization by N. M. Katz, and \cite{[vdP]}.}


 It is a fact (due to T. Honda and N. M. Katz) that a full set of polynomial solutions modulo $p$ is equivalent to having   a full set of power series solutions in $\F_p[[x]]$ (see \cite{[H]}, or also  \cite{[DGS]}, III.2). A simple argument shall be given in \S \ref{SSS.honda-katz} below. 
 
 So, {\it a fortiori},  the Grothendieck conjecture yields that  {\it integrality  for the Taylor coefficients of a full system of solutions of a linear differential equation over $\overline\Q(x)$ may come only from algebraicity of the solutions}. \footnote{ Conversely, the Wronskian criterion shows that linear independence over the constant field is preserved by reduction modulo $p$, for all sufficiently large primes $p$. Hence, for such primes, the reduction of a full system of power series solutions with $p$-integer coefficients yields a full system of power series solutions modulo $p$.}

 We see that the integrality of the $c_n$ (for odd primes) is no coincidence:  it could not hold if the generating series  would not be algebraic, because this series satisfies a differential equation of order $1$ over $\Q(x)$, and  the Grothendieck conjecture is known in these cases.\footnote{ This case of the conjecture is not too difficult, on using the Cebotarev theorem to prove rationality of the residues; see   \cite{[vdP]}, \S 3.1. It immediately follows also from a theorem of Katz on `global nilpotence', see \cite{[DGS]}, Thm. 6.1. We further point out that `order $1$' is crucial here; indeed, there are transcendental series having integer coefficients and satisfying linear differential equations over $\Q(x)$; well-known examples arise with hypergeometric functions (see e.g. \cite{[H]}, \S 4) and with {\it Ap\'ery numbers} (see \cite{[vdPo]}).} The same holds for the congruence properties: by the said {\it converse}, these may hold for all odd primes only if $s(x)$ is of the shape $xf'(x)/f(x)$  for a series $f(x)$  with $p$-integral coefficients (at odd primes $p$);  but  $f$ then satisfies the differential equation, again  of the first order, $xf'=sf$; hence by the Grothendieck conjecture (proved  in this special case by Andr\'e as Prop. 6.2.2 in \cite{[A2]} \footnote{ This is much more difficult than the previously quoted result because now the equation is defined over a non-rational function field extension of $\Q(x)$}), $f$ must be algebraic as well,  
which is indeed what happens in our example.

We note that in our  arguments for Theorem \ref{T.congr}  we do not need recourse to any case of the Grothendieck conjecture: we only need  algebraicity properties which may be checked directly. Rather, the conjecture confirms that this algebraicity  is {\it a priori}  the only possible motivation for  the assertions in the theorem. On the contrary,  in the proof of Theorem \ref{T.converse} we shall use some cases of the said conjecture.

\medskip

\underline{\it Theorem  \ref{T.converse}} : As for Theorem \ref{T.congr}, to rephrase in convenient shape the assertions, one  starts by showing that the generating function $S(x)$ of any solution of (\ref{E.recurrence}) is a solution of  a certain  linear differential equation over $\Q(x)$, this time  inhomogeneous, whose homogeneous part is  the  same  already appearing in the proof of Theorem \ref{T.congr}. In turn, $S(x)$  is a solution of a corresponding (homogeneous) linear differential equation of the second order. It shall easily turn  out that, except for the 
special cases when $S(x)=a+bs(x)$, it is a transcendental function (Lemma \ref{L.algebraic}).

Now, the first part of the theorem is a consequence of a well-known formula which expresses `by quadratures' \footnote{ By this we mean that the expression involves only taking integrals of `known' functions.}  the solutions of an inhomogeneous
 linear differential equation of the first order  in terms of any given nonzero solution of the homogeneous one. 
 This also indicates that if a second order linear differential equation over $\Q(x)$ has an algebraic solution, then even the possible transcendental solutions have coefficients whose denominators are in a sense smaller  than may be expected. 
  The mentioned example (see a footnote above) of the coefficients of $\exp(x/(1-x))$, whose denominators grow like factorials, may be explained, e.g., by the fact that the associated first order differential equation  has irregular singularities. We refer to \cite{[DGS]}, Ch. III, for work of Honda and Katz exploiting the relevance of this kind of condition towards our issue. Also, we wonder here if a suitable conjecture can be stated predicting exactly, and in general,  when this semi-integrality phenomenon arises for a complete system of solutions of a linear differential system.\footnote{ Interesting examples of full integrality arise again with 
 Ap\'ery numbers; see \cite{[vdPo]} and also \cite{[B]}.} In this direction, the theory of $G$-functions provides highly important results which link the `average' $p$-adic radii of convergence of the solutions (which are naturally  related to the denominators of the coefficients) to the behaviour of the reductions modulo $p$ of the  differential equation; 
 we refer to \cite{[A]} and \cite{[DGS]} (especially Chs. VII, VIII) for an account of this theory. 
 
  

The second part of Theorem \ref{T.converse} is again strictly related to Grothendieck's conjecture, actually to  another special case of it, proved by the Chudnovski brothers in \cite{CC}, and by Andr\'e as Prop. 6.2.1 of \cite{[A2]} (see that paper also for further references), and roughly speaking stating that {\it if a differential form on an algebraic curve is exact when considered modulo $p$, for (almost) all primes $p$, then it is exact.} 

Similarly to the above, this yields that  {\it if  the integral of a power series representing an algebraic function has integral coefficients, then this integral is also algebraic}.  This attractive statement appears   already at p. 149 of \cite{[A]}, where it is described as a generalization of a simple but elegant result by P\'olya, who indeed  proved  in \cite{[P]} the assertion for the special (and extraordinarily  easier) case of rational functions (briefly discussed in \S \ref{SS.polya}) and raised elsewhere the case of general algebraic functions as an open question. 

This result by Andr\'e is in turn merely a special case of more general assertions of his, whose proofs rely among other things on delicate  techniques from transcendental number theory  (applied to $G$-functions) and on a uniformization procedure first suggested and implemented by D. and G. Chudnovsky in \cite{CC}. An application of Andr\'e's result  to our context suffices to exclude integrality of coefficients for any $S(x)$ as above, apart from those with `special' initial  coefficients.






\medskip

\underline{\it Theorems  \ref{T.modp} and \ref{T.union}}. These results shall be deduced using mainly the Cartier operator, whose main properties shall be recalled in \S \ref{SS.cartier}. (Complete self-contained proofs may be found in Appendix 2 of \cite{[L]}.) We shall also use a simple but important result of Honda and Katz, for which we shall provide a quick self-contained proof in \S \ref{SS.polya}. 

More precisely, we shall analyze the action of the Cartier operator on the space of differential forms of the first and second kind on a certain elliptic curve over $\F_p$; this action is regulated by the trace of Frobenius and by another integer modulo $p$. We haven't been able neither to locate in the literature any study on this, nor to discover the precise relation between these invariants, as the curve is the reduction of a given curve over $\Q$ and $p$ varies. But we have proved at least that these invariants cannot both vanish. For this we shall provide two different arguments; although one of them is related to well-known elementary facts on the Hasse invariant, we haven't found any reference for this in the literature. This result shall imply the assertion on the dimension  of $V_p$, whereas the rest of Theorem \ref{T.modp} shall be derived more easily in \S \ref{SS.polya}, exploiting the vanishing of a certain residue.


The proof of Theorem \ref{T.union} shall be derived by combining similar arguments with the analysis of an unramified cover of the relevant elliptic curve modulo $p$.

\bigskip

We shall often pause to illustrate and recall certain well-known facts, useful for or related to the main issues. We shall also indicate suitable literature; for instance,  the arithmetic of linear differential equations is discussed in the books \cite{[A]} and  \cite{[DGS]}; the notion of $p$-curvature and the Grothendieck conjecture are introduced and discussed in some simple cases   in the  introductory paper \cite{[vdP]}, or, more deeply, in \cite{[A2]}, where the conjecture is proved in some cases.

The paper shall be concluded with a section illustrating some related facts and several open questions.

We heartily thank Yves Andr\'e and Nicholas M. Katz for many important comments and references. We also thank   Jean-Paul Bezivin for the reference to George P\`olya's article \cite{[P]}.

\section{Some congruences modulo $p$ for coefficients of power series} \label{S.congr}

We prove two very simple but, at first sight, somewhat surprising results. 
In the proofs, for $a\in \Z_p$, we interpret $(1+x)^a$ as the formal power series $\sum_{n=0}^\infty 
{a\choose n}x^n$. Note that this lies in $\Z_p[[x]]$ and  satisfies the properties $(1+x)^a(1+x)^b=(1+x)^{a+b}$, as follows from the universal identities on binomial coefficients. We use the symbol `$O(x^N)$' to mean a power series containing only terms of order $\ge N$.

\medskip

\begin{lem} \label{L.dieu} Let $f\in \Z_p[[x]]$ be such that $f(0)=1$. Then there exist  unique  $a_1,a_2,\ldots \in \Z_p$ such that $f(x)=\prod_{m=1}^\infty(1-x^m)^{a_m}$.
\end{lem}

\begin{proof} We prove that for $n\in\N$ there are $a_1,\ldots ,a_n\in\Z_p$ such that $f(x)=\prod_{m=1}^n(1-x^m)^{a_m}+O(x^{n+1})$; this clearly amounts to the sought conclusion. 

For this last statement, we argue  by induction on $n$. For $n=0$ this is the assumption $f(0)=1$, 
so suppose $n>0$ and the conclusion true up to $n-1$. By the induction hypothesis we may write $f(x)=\prod_{m=1}^{n-1}(1-x^m)^{a_m}-ax^n+O(x^{n+1})$, for some $a\in\Z_p$.
 Hence, using $1-ax^n=(1-x^n)^a+O(x^{n+1})$, we obtain 
\begin{equation*}
f(x)=\prod_{m=1}^{n-1}(1-x^m)^{a_m}-ax^n\prod_{m=1}^{n-1}(1-x^m)^{a_m}+O(x^{n+1})=
\left(\prod_{m=1}^{n-1}(1-x^m)^{a_m}\right)(1-x^n)^a+O(x^{n+1}),
\end{equation*}
and the induction step follows on setting $a_n:=a$, concluding the argument. The uniqueness of the $a_m$ is also clear in view of the same formulae.
\end{proof}

\medskip

\begin{rem}\label{R.dieu} This is similar, but not quite equal, to a lemma of Dieudonn\'e, appearing as Lemma 5.1 of [DGS] (and leading to `Dieudonn\'e's theorem', as stated therein); in that lemma the factors 
$(1-x^m)^{a_m}$ are replaced by factors $(1-b_mx^m)$. 

Also, as in that case, the lemma   may be stated for a `universal' series $f(x)=1+f_1x+f_2x^2+\ldots$, where the $f_m$ are indeterminates over $\Z$, and then  the $a_n$ turn out to be   certain universal polynomials in $\Q[f_1,f_2,\ldots ]$; the point here is that for $f_i\in\Z_p$ the values of these polynomials lie in $\Z_p$ as well.

\end{rem}

The following proposition, in a  different formulation, and with a different proof,  appears e.g. as the `Proposition' at p. 143 of  \cite{[B]}; see \cite{[B]} also for references to further proofs. The proof given below, differently from other ones, does not use Fermat's Little Theorem, which instead shall be found as a consequence (actually using only $\Z$ in place of  $\Z_p$), as anticipated above.

\medskip

\noindent{\bf Notation.} We shall use the following notation, also in the sequel. For a formal series $s(x)\in \O [[x]]$, where $\O$ is a domain  with quotient field $k$, we let ${\mathcal L}(s)\in k((x))$ denote the formal logarithmic derivative, i.e. $s'(x)/s(x)$. If $s(0)=1$  and $s\in\O[[x]]$ then ${\mathcal L}(s)$   lies in $\O[[x]]$, as follows on writing $s(x)=1-xg(x)$, where $g\in \O[[x]]$ and $1/s(x)=1+xg(x)+x^2g(x)^2+\ldots$, which is a convergent expansion in $\O[[x]]$. 

This operation satisfies  ${\mathcal L}(s_1s_2)={\mathcal L}(s_1)+{\mathcal L}(s_2)$, as follows from the easy formal Leibniz rule $(s_1s_2)'=s_1's_2+s_1s_2'$. Also, we have ${\mathcal L}((1+x)^a)={a\over 1+x}$.

\medskip

\begin{prop}\label{P.congr}   Let $f\in \Z_p[[x]]$ be such that $f(0)=1$ and write $xf'(x)/f(x)=\sum_{n=1}^\infty c_nx^n$. Then $c_n\in\Z_p$ for all $n\in\N$ and, for every $r,k\in\N$,  $c_{kp^{r+1}}\equiv c_{kp^r}\pmod{p^{r+1}\Z_p}$.
\end{prop}


\begin{proof} That the $c_n$ lie in $\Z_p$ is clear from the above remarks.

 We now apply Lemma \ref{L.dieu} to $f(x)$, obtaining $f(x)=\prod_{m=0}^\infty (1-x^m)^{a_m}$, where $a_m\in\Z_p$, an identity valid in $\Z_p[[x]]$. Taking ${\mathcal L}$  of both sides, using the rules just recalled (which clearly extend to an infinite convergent product) and multiplying by $x$, we obtain
 \begin{equation*}
 \sum_{n=1}^\infty c_nx^n=-\sum_{m=0}^\infty {ma_mx^m\over (1-x^m)}.
 \end{equation*}
 The series on the right may be in turn expanded as 
 \begin{equation*}
 -\sum_{m=0}^\infty ma_mx^m\sum_{l=0}^\infty x^{ml}=-\sum_{l,m\ge 0}ma_mx^{m(l+1)}=-\sum_{n=0}^\infty (\sum_{m|n}ma_m)x^n.
 \end{equation*}
 Hence we have
\begin{equation*}
 c_n=-\sum_{m|n}ma_m.
 \end{equation*}  
 This implies that $c_{kp^{r+1}}-c_{kp^r}$ is the sum of terms $-ma_m$ where $m$ divides $kp^{r+1}$ but not $kp^r$; but then all such terms are multiples of $p^{r+1}$, and we immediately obtain the desired congruence.

\end{proof}

\medskip

\begin{rem}\label{R.composite}  
 
 (i) Another proof of (at least)   the case $r=0$ of the lemma may be quickly obtained  by using the so-called {\it Cartier operator}; we shall sketch  the argument in \ref{E.modp}. 
 
(ii) We sketch still another (self-contained) argument for the case $r=0$. Let us define a derivation on  $\Z[T_1,T_2,...]$ by  $D(T_i)=T_{i+1}$  and polynomials $Q_m\in \Z[T_1,T_2,\ldots ]$  by  $Q_0=1$, $Q_1=T_1$  e $Q_{m+1}=(D+T_1)Q_m$.  These polynomials express successive derivatives of $z=f'/f$: we have indeed $f'=zf$, $f''=(z'+z^2)f$ and in general $f^{(n)}=Q_n(z,z',\ldots ,z^{(n-1)})f$. In view of the additivity of the logarithmic derivative these polynomials satisfy the formula
 \begin{equation*}
Q_n(T+U)=\sum_{r+s=n}{n\choose r} Q_r(T)Q_s(U).
\end{equation*}
where $T=(T_1,T_2,...)$ and similarly for $U$. (It suffices to consider $z=f'/f$ and $w=g'/g$.) For $n=p$  we then have 
\begin{equation*}
Q_p(T+U)\equiv Q_p(T)+Q_p(U),
\end{equation*}
whence one easily deduces $Q_p\equiv T_p+T_1^p$. Hence $f^{(p)}(x)\equiv (z^{(p-1)}+z^p)f(x)$. But if $f$ is a series with  integer coefficients we have $f^{(p)}\equiv 0\pmod p$, 
and it follows that $c_{kp}\equiv c_k\pmod p$. (All of this is   related to the notion of `$p$-curvature'; see \cite{[A2]}, p. 63, and also \cite{[vdP]} for an elementary discussion.)

(iii) We could have used also the lemma of Dieudonn\'e mentioned in Remark \ref{R.dieu}.
Along  the same lines we would obtain $c_n=-\sum_{m|n}mb_m^{n/m}$. This is slightly more complicated but has the advantage that one can obtain congruences over fields other than $\F_p$. For instance, over any domain $\O$ we obtain that $c_{kp}\equiv c_k^p\pmod {p\O}$. 
\end{rem}

\medskip

\subsection{A deduction of Fermat's Little Theorem.}  We note that Proposition \ref{P.congr}  immediately implies Fermat's Little Theorem that $a^p\equiv a\pmod p$ for an integer $a$. It suffices to consider the function   $f(x)=1-ax$, so $xf'(x)/f(x)=-ax/(1-ax)=-\sum_{n=1}^\infty a^nx^n$. Taking $k=1, r=0$ in Lemma 2 yields the result.

\medskip

\noindent{\it 
Methodological remark.}  Concerning this deduction, 
we note that the given argument, though certainly much more complicated than the usual ones, seems logically different, and for instance does not use directly the unique factorization in $\Z$. 
(Implicitly, it exploits  that the derivative of $(1+x)^p$ vanishes  in characteristic $p$.)

\medskip

\subsection{A converse proposition.}\label{SS.converse}  Following \cite{[B]}, we point out that Proposition \ref{P.congr} has a relevant   kind of converse, namely: 

\smallskip

{\it If $c_1,c_2,c_3,\ldots$ is a sequence in $\Z_p$ such that 
 for every $r,k\in\N$,  we have  $c_{kp^{r+1}}\equiv c_{kp^r}\pmod{p^{r+1}\Z_p}$, then $\exp(\sum_{n=1}^\infty {c_n\over n}x^n)$ is a power series in $\Z_p[[x]]$. 
 
 In other words, there exists $f\in \Z_p[[x]]$ such that $f(0)=1$ and $\sum_{n=1}^\infty c_nx^n=x{f'(x)\over f(x)}$.} 
 
 \medskip
 
 The last formulation illustrates  the link with the above. 
 
 For a proof of this converse see \cite{[B]}, Proposition at p. 143 (also for references to other possible arguments). A  proof also follows just on reversing the arguments for Proposition \ref{P.congr}; indeed, the said congruences easily imply that $\sum_{m|n}\mu(m)c_{n/m}$ is divisible by $n$ in $\Z_p$ (where $\mu$ is the M\"obius function). And now it suffices to define $a_n:=-(\sum_{m|n}\mu(m)c_{n/m})/n$, $f(x):=\prod_{n=1}^\infty (1-x^n)^{a_n}$ and to note that $c_n=-\sum_{m|n}ma_m$.

 See \cite{[B]} also for  an interpretation of these congruences in terms of formal groups (over $\Z_p$ or $\Z$).
 
 \medskip
 
 To conclude this short subsection, we also point out that:
 
 {\it  The weaker congruences $c_{kp}\equiv c_k\pmod p$ (for a given $p$ and $k\in\N$), i.e. the case $r=0$ of the above, imply that the reduction modulo $p$ of $\sum_{n=1}^\infty c_nx^n$ is of the shape $x{f'(x)\over f(x)}$, where now $f(x)\in \F_p[[x]]$.}
 
 This   assertion may be derived using the Cartier operator (see \S\ref{SS.cartier}), but is also a direct consequence of what has just been proved. In fact, define a sequence $(\tilde c_n)_{n\in\N}$ in $\Z_p$ by $\tilde c_n:= c_r$ where $n=p^mr$ with integers $m,r\ge 0$ and $r$ prime to $p$. Then the congruences $c_{pk}\equiv c_k\pmod p$ imply  $\tilde c_n\equiv c_n\pmod p$. Also, $\tilde c_{pm}=\tilde c_m$ so the stronger congruences are true for the new sequence. Then the generating function $\sum_{n=1}^\infty \tilde c_nx^n$ equals $x{\tilde f'(x)\over \tilde f(x)}$ for a suitable $\tilde f\in\Z_p[[x]]$. But then, reduction modulo $p$ yields what needed.

\bigskip

\subsection{An elliptic example.} 
Let $y^2=F(x)$ be a  Weierstrass equation of an elliptic curve $E$, so $F(x)$ is a   cubic polynomial without multiple roots. We suppose that $F(x)$ is monic and lies in $\Z_p[x]$. Let $\omega={\d x\over y}$ be an invariant differential, and let $t=x/y$ be a local parameter at the origin of $E$ (i.e. the point at infinity). We may expand $t^2x$ as a series  in $t$: $t^2x=s(t)=1+s_1t+\ldots \in \Z_p[[t]]$ (see e.g. \cite{[L]}, Appendix 1).      We have $\omega=t\d x/x=-2\d t+ts'(t)\d t/s(t)$  
Then Proposition \ref{P.congr}   yields obviously a congruence for the coefficients of $\omega$ in its $t$-expansion. 

Other   somewhat similar (but more significant) congruences for the coefficients of $\omega$ were originally found by Atkin and Swinnerton-Dyer; they relate the elements   $(c_{p^rn+1}$ rather than the $c_{p^rn}$: 
see for instance \cite{[L]}, Appendix 2, p. 313, for a  derivation of them  using the Cartier operator, similarly to Remark \ref{R.composite}(i) above, and see \cite{[Ha]}, VI, 33.2,  for the case of higher powers of $p$ (where the congruences have   different  shape   with respect to the above ones).


\medskip

\section{An amusing application to a proof of Theorem \ref{T.congr}} Similarly to the last examples, 
we shall now give  an application  of these congruences to the coefficients of the series expansions of a suitable algebraic function constructed from the recurrence, proving Theorem \ref{T.congr}. 
We believe that   such applications, though elementary, are  more hidden and surprising than, e.g., the one to Fermat's Little Theorem which we have just shown.

\medskip

\begin{proof}[Proof of Theorem \ref{T.congr}.]  We try  to give the proof in a {\it pedagogical} way, which shall slightly increase the length. 
Simultaneously, it shall also be clear how to construct easily similar (and more recondite) examples.

\medskip

Clearly the recurrence (\ref{E.recurrence}) and the data (\ref{E.initial})  define uniquely a sequence $(c_n)_{n\in\N}$ of rational numbers. Let us then consider the generating function 
\begin{equation*}
s(x):=\sum_{n=0}^\infty c_nx^n\in \Q[[x]], 
\end{equation*}
so $xs'(x)=\sum_{n=0}^\infty nc_nx^n$.  The recurrence may be also written as 
\begin{equation*}
4((n+5)-1)c_{n+5}+8((n+4)-2)c_{n+4}+(n+3)c_{n+3}+(4(n+2)-1)c_{n+2}+(5(n+1)-1)c_{n+1}+2nc_n=0. 
\end{equation*}
Now,
\begin{equation*}
 (4+16x+x^3+x^4)s(x)=\sum_{n=0}^\infty  (4c_n+16c_{n-1}+c_{n-3}+c_{n-4})x^n,
\end{equation*}
where by convention $c_m:=0$ if $m<0$. 
Similarly
\begin{align*}
(4+8x+x^2+4x^3+5x^4+2x^5)xs'(x)= & \\ \sum_{n=0}^\infty \left(4nc_n+8(n-1)c_{n-1}+(n-2)c_{n-2}+4(n-3)c_{n-3}+5(n-4)c_{n-4}+2(n-5)c_{n-5}\right)x^n.\\
\end{align*}
Hence
\begin{align*}
(4+8x+x^2+4x^3+5x^4+2x^5)xs'(x)-(4+16x+x^3+x^4)s(x)= & \\ 
\sum_{n=0}^\infty \big(4(n-1)c_n+8((n-1)-2)c_{n-1}+ &\cr (n-2)c_{n-2}+  (4(n-3)-1)c_{n-3}+(5(n-4)-1)c_{n-4}+2(n-5)c_{n-5}\big)x^n.\\
\end{align*}
Of course, by the recurrence \eqref{E.recurrence}, in the right hand side we need only keep the terms with $n=0,1,2,3,4$. In other words

\begin{align*}
(4+8x+x^2+4x^3+5x^4+2x^5)xs'(x)-(4+16x+x^3+x^4)s(x)= &\\
 -4c_0-16c_0x+(4c_2-8c_1)x^2+(8c_3+c_1-c_0)x^3+ (12c_4+8c_3+2c_2+3c_1-c_0)x^4.\\
 \end{align*}
 This identity holds for general solutions of the recurrence \eqref{E.recurrence}. Note that the coefficients on the right are five linear forms in $c_0,\ldots ,c_4$, spanning a space of dimension $4$ (not $5$ ! \footnote{ This corresponds to the fact that the differential equation below has only one solution up to constant factors.} ). Also,  the above initial data show that each coefficient on the right vanishes (i.e., \  the full sequence $(c_n)_{n\in\Z}$ does  satisfy the said recurrence), hence for the actual series $s(x)$ we have
\begin{equation}\label{E.diff.eq.}
(4+8x+x^2+4x^3+5x^4+2x^5)xs'(x)-(4+16x+x^3+x^4)s(x)=0,
\end{equation}
equivalently
\begin{equation*}
{s'(x)\over s(x)}={4+16x+x^3+x^4\over x(4+8x+x^2+4x^3+5x^4+2x^5)}.
\end{equation*}
Since on the left we have a logarithmic derivative, it is natural to try to understand the residues of the rational function on the right, and in turn one tries to factor the denominator. One finds
\begin{equation*}
4+8x+x^2+4x^3+5x^4+2x^5=(1+2x)(4+x^2+2x^3+x^4)=(1+2x)Q(x),
\end{equation*}
say, where the last factor
\begin{equation*}
Q(x):=4+x^2+2x^3+x^4=4+(x+x^2)^2
\end{equation*}
  turns out to be irreducible over $\Q$. As to partial fractions, one finds
\begin{equation*}
{4+16x+x^3+x^4\over x(4+8x+x^2+4x^3+5x^4+2x^5)}={1\over x}+{1\over x+{1\over 2}}-{x(1+3x+2x^2)\over 4+x^2+2x^3+x^4}.
\end{equation*}
However $Q'(x)=2x(1+3x+2x^2)$.  Hence the function $s(x)$ satisfies 
\begin{equation*}
{\mathcal L}(s(x))={\mathcal L}(x)+{\mathcal L}(x+{1\over 2})-{1\over 2}{\mathcal L}(Q(x)),
\end{equation*}
whence
\begin{equation*}
{\mathcal L}\left({s(x)^2Q(x)\over x^2(x+{1\over 2})^2}\right)=0,
\end{equation*}
so
\begin{equation*}
s(x)^2=c{x^2(x+{1\over 2})^2\over Q(x)},
\end{equation*}
for a constant $c$, which is readily found to be $16$, leading to 
\begin{equation*}
s(x)^2= {4x^2(2x+1)^2\over Q(x)},
\end{equation*}

Let now  $y$ be a square root of $Q(x)$ in some extension field of $\Q(x)$; since $Q(0)=4$ is a square in $\Q$, we see that $y$ may be realized as a series $y(x)$ in $\Q[[x]]$, chosen so that $y(0)=2$. Comparing with the series for $s(x)=x+O(x^2)$ we finally have
\begin{equation}\label{E.s(x)}
s(x)= {2x (2x+1)\over y(x)} \in \Q(x,\sqrt{Q}).
\end{equation}
This shows that, in particular, $s(x)$ is an algebraic function in $\Q[[x]]$.  Note that $s(x)$ is not rational, so the $c_n$ do not satisfy any  linear recurrence with constant coefficients \footnote{ It is indeed an easily proved  well-known fact that the solutions of such recurrences yield rational generating functions.}, as asserted.  

Also, being algebraic, its coefficients can have only finitely many prime numbers in their denominators (Eisenstein Theorem, see e.g.  [DvdP]). Actually, a binomial expansion of $y(x)=2\sqrt{1+{1\over 4}x^2+{1\over 2}x^3+{1\over 4}x^4}=2\sum_{m=0}^\infty{1/2\choose m}4^{-m}x^{2m}(1+2x+x^2)^m$ readily shows that in fact only the prime $2$ can occur in the denominators of the $c_n$, proving another assertion  of the theorem. 

For an odd prime $p$ we may then  consider the reduction of $s(x)$ modulo $p$.  A well-known theorem of Ostrowski (see, e.g., \cite{[Z]}), implies that it remains non-rational for all large primes. In the present simple case it is indeed very easy to check that this happens for all odd primes, since then $Q(x)$ is not a square modulo $p$; in turn, this yields at once that the sequence $(c_n)$ is not eventually periodic modulo an odd prime (as remarked after the statement).

\medskip

\begin{rem} Later in \S \ref{SS.explicit}, using a similar binomial expansion,  we shall exhibit  an `explicit' formula for the coefficients $c_n$. However such expression (which anyway is very special of the present data) seems not helpful for proving the above congruences and the full theorem. (We may be wrong, and maybe such a proof  exists and is not too complicated; we believe that, although special, a direct proof using such explicit formulas would be interesting; we leave this problem for the interested reader; we also refer to \cite{[B]} for some direct proofs in this sense.)
\end{rem}

\medskip

\subsection{An elliptic curve.} \label{SS.elliptic} We note that the field $\Q(x,\sqrt{Q})$ is an elliptic function field, because $Q$ has no repeated roots. We let $E/\Q$ be a complete nonsingular curve with such a function field. It is of genus $1$ and defined (birationally) by the (non-Weierstrass) equation $Y^2=Q(X)$. \footnote{ Note that this equation defines a curve in $\P_2$ which is singular at infinity, but we interpret $E$ as a non-singular model. See, e.g., \cite{[Sil]}, Ch. 2, for a smooth embedding in $\P_3$.} The series $y(x)$ represents a Puiseux expansion of the rational function $y$ at one of the two points of $E$ above $x=0$, i.e. the point $(0,2)$ where $y$ has value $2$. The series $s(x)$ represents the (expansion at $(0,2)$ of the) rational function $2x(2x+1)/y$ on $E$, denoted $s$ in the sequel.

At infinity there are two expansions, corresponding to two places  of $E$, denoted $\infty_{\pm}$, and given by $y=\pm x^2\sqrt{1+{2\over x}+{1\over x^2}+{4\over x^4}}$, where it is understood that the square root is expanded with the binomial theorem as $1+{1\over x}+\ldots$. This also shows that $\infty_\pm$ are defined over $\Q$. We may choose the point $\infty_+$ as the origin for a group law on $E$, which becomes thus an elliptic curve over $\Q$. 

A Weierstrass equation for $E$ (whose point at infinity corresponds to $\infty_+$) is obtained by setting $z=x^2+x+y$ and observing that the divisor of $z$ is $2(\infty_-)-2(\infty_+)$. Putting also $t:=1+2x$, Also, we have $(2tz)^2=(2z)^3+(2z)^2-16(2z)$, which provides a Weierstrass equation. If we want to eliminate the quadratic term on the right, we may set $U=2z+{1\over 3}$, $V=2tz=t(U-{1\over 3})$, and obtain  $V^2=U^3-{49\over 3}U+{146\over 27}$. The $j$-invariant of $E$ is computed as $j(E)={7^6\over 5\cdot 13}$. In particular, since this is not an integer, $E$ has not complex multiplication.

\medskip

We also observe that the zeros of $t$ (which are defined over a field obtained by adding $\sqrt{65}$ to the ground field) are points of order $2$; the other point of order $2$ is $\infty_-$. The operation $(x,y)\mapsto (x,-y)$ corresponds to the involution $\zeta\mapsto \infty_--\zeta$ on $E$. The two zeros of $x$ are invariant under this involution; they correspond to $(0,\pm 2)$ in the $xy$-model and to $(\pm 4,\pm 4)$ in the first of the above Weierstrass models. On computation on this model shows that tripling $(4,4)$ leads to a point with non-integer coordinates; hence from the Lutz-Nagell theorem (see \cite{[Sil]}) we deduce that none of these points is torsion on $E$.

\medskip

For later reference, we now exhibit some relevant differential forms on $E$. First, set $\omega:={\d x\over y}$. We claim that it is {\it of the first kind}, i.e. regular on $E$. Indeed, at finite points $\alpha\in E$, $x-x(\alpha)$ is a local parameter, unless $x(\alpha)$ is a zero of $Q$; in the first case, $y(\alpha)\neq 0$ and $\omega$ is regular at $\alpha$; in the second case, $y(\alpha)=0$, $y$ is a local parameter and we may use $2y\d y=Q'(x)\d x$ to write $\omega =2\d y/Q'(x)$. But $Q$ has no multiple zeros, so $Q'(x(\alpha))\neq 0$, proving the claim. In the case of a point at infinity $\infty_\pm$, $u:=1/x$ is a local parameter, and $\omega=-\d u/u^2y$; this shows that $\omega$ is regular also at infinity, since $y$ has a pole of order $2$ there. It is a general fact of elliptic curves that $\omega$ is unique up to multiplication by a nonzero constant; it also follows that it is invariant for any group translation.

Further, consider the form $\eta:=(x^2+x)\omega$; this is not regular (it has poles at $\infty_\pm$), however it is {\it of the second kind}, i.e. with no residues. To check this, we have only to consider its expansion at the poles, i.e. $\infty_\pm$; as before, $u=1/x$ is a local parameter there, and $\eta=-(u+1)\d u/u^4y$. But the expansion of $u^4y$ is of the shape $\pm u^4x^2\sqrt{1+2u+u^2+4u^4}=\pm u^2(1+2u+u^2+...)^{1/2}$. Hence $\eta=\pm u^{-2}(u+1)(1+2u+O(u^2))^{-1/2}\d u=u^{-2}(1+O(u^2))\d u$, so  there is no   term in $u^{-1}$, proving the contention. 

Now, every exact differential form (i.e., of the shape $\d g$ for a rational function $g$ on $E$) is clearly of the second kind; and $\omega,\eta$ are also of the second kind. It turns out that {\it  in characteristic zero} they are linearly independent modulo exact forms: if $a\omega+b\eta=\d g$, for constants $a,b$ and $g\in\C(E)$, then $g$ could have poles only at $\infty_\pm$, moreover of order at most $1$.    But then $g$ would be necessarily of the shape $cx+d$ for constants $c,d$ (e.g. by the theorem of Riemann-Roch), leading to $a+b(x^2+x)=cy$, whence  $a=b=c=0$, as required. \footnote{ Observe that in characteristic $p$ this does not hold; for instance if $E$ is an elliptic curve with supersingular reduction, then $\omega$ is exact (see the Appendix 2 in \cite{[L]}). Actually, we shall see in \S \ref{SS.cartier} that they are never independent. The above deduction on the poles of $g$ may not hold because its poles could disappear by differentiation.}  It is a rather fundamental fact of the theory (see [L2], Theorem 8.1) that the dimension of the vector space of the forms of the second kind modulo exact ones is twice the genus, hence $2$ in the present case; this proves that this quotient space is generated by the images of $\omega$ and $\eta$. See the quoted book \cite{[L2]}  for other simple background facts which we tacitly assumed.

\smallskip

Finally, we add a few words on {\it good reduction}, which we interpret in a naive sense, sufficient for the present purposes. If $p$ is a prime number we can consider the reduction modulo $p$ of an equation for $E$, for instance $y^2=Q(x)$. This reduction yields an elliptic curve provided the reduction of $Q(x)$ does not have repeated roots, and we then speak of good reduction. To find when this occurs, observe that the repeated  roots are solutions of the system $Q(x)=Q'(x)=0$, and it turns out very easily that  for this to be solvable 
we must have $p=2,5,13$. Hence for other primes we have good reduction. 

We shall   consider in the sequel the function field $k(E)$ of the reduced curve, where $k$ is an algebraic extension of $\F_p$.  The Frobenius map $u\mapsto u^p$ sends $k(E)$ into the subfield $k(E)^p$; the extension $k(E)/k(E)^p$ is purely inseparable and we contend that $[k(E):k(E)^p]=p$. This is well known and in any case follows  on considering the four fields $k(x),k(x^p),k(E),k(E)^p$ and using the immediate observations  that $[k(x):k(x^p)]=p$  and that $k(E)^p=k(x^p,y^p)$, so $[k(E)^p:k(x^p)]=2$. (See also \cite{[L]}, Ch. IX, \cite{[Sil]}, Ch. II, especially Prop. 2.11, and see Ch. VII.5 for properties of good reduction.)

\bigskip
\medskip

Let us now go ahead with the proof of Theorem 1. It makes sense to consider the differential form $\xi:=s{\d x\over x}={2(2x+1)\d x\over y}=2(2x+1)\omega$ on $E$. \footnote{ This consideration seems rather unmotivated {\it a priori}, i.e. seems  to come from nowhere; however some motivation comes from the above Proposition \ref{P.congr}, especially from its converse mentioned in \ref{SS.converse},  and further  from the structure of residues of  $\xi$. Actually, as remarked in \ref{SS.comments}, if we knew {\it a priori} the validity of the congruences to be proved, then   \ref{SS.converse}  above together with a deep theorem of Andr\'e 
would imply that $\xi$ is the logarithmic differential of an algebraic function; this is indeed what we are going to check in (\ref{E.log}) below.}  
From the fact that $\omega$ is everywhere regular on $E$, it follows that $\xi$ has  poles only at $\infty_\pm$, and an easy calculation, using that $1/x$ is  a local parameter there,  shows that the residues therein are given by $\mp 4$ respectively. The residues being integers, it makes sense to ask whether $\xi$, or perhaps an integer multiple $m\xi$,  is of the shape $\d v/v$ for a function $v$ on $E$. Such a possible function should have zeros/poles only at infinity, with order $2m$. Now, this reminds of the Pell equation in polynomials, and here we have an example given by the equality $(x^2+x)^2-Q(x)\cdot 1^2=4$, leading to $(x^2+x+y)(x^2+x-y)=4$. This identity exhibits that the function $z:=x^2+x+y$ (which satisfies $z^2-2(x^2+x)z-4=0$) has a pole of order $2$ at $\infty_+$ and a zero of order $2$ at $\infty_-$ and no other zeros/poles, hence is a perfect candidate. Indeed, we may directly verify  that
\begin{equation}\label{E.log}
\xi={s\cdot \d x\over x}=2{\d z\over z}.
\end{equation}
This formula is checked as follows. We have $\d z=(2x+1+{\d y\over \d x})\d x$, hence $x\d z=x(2x+1+{\d y\over \d x})\d x$, and it now suffices to verify that $x(2x+1)+x{\d y\over \d x}={z s\over 2}$, which in turn amounts to $2yx(2x+1)+yx{Q'(x)\over y}=2(x^2+x+y)x(2x+1)$, i.e. $xQ'(x)=2x(2x+1)(x^2+x)$, which is indeed true.

\medskip

For $p>2$ we may expand $z$ at the said point $(0,2)\in E$ as a series in $\Z_p[[x]]$;  indeed, it suffices to use the  expansion $y(x)^{-1}={1\over 2}\left(\sqrt{1+{1\over 4}x^2+{1\over 2}x^3+{1\over 4}x^4}\right)^{-1}={1\over 2}\sum_{m=0}^\infty{-1/2\choose m}4^{-m}x^{2m}(1+2x+x^2)^m$. 
Then Proposition \ref{P.congr}, with $z(x)/2$ in place of $f(x)$,  applies to $s(x)/2$, proving the second half of Theorem \ref{T.congr}.

\end{proof}

\medskip

\subsection{An `explicit' formula.} \label{SS.explicit}  As anticipated above, we now  pause to give a kind of explicit formula for the $c_n$, and add some brief remarks on it.  We have already noted that if $s(x)$ would be rational the $c_n$ would admit the usual `explicit' expression as the values at integers of an exponential polynomial. This is not the present case, however there exists even in this case some `explicit' expressions for this. 

It is perhaps convenient to replace the $c_n$ with $l_n:=2^{2n-2}c_n$. The recurrence for the $c_n$ leads to the recurrence

\begin{equation*}
 (n+4)l_{n+5}+8(n+2)l_{n+4}+4(n+3)l_{n+3}+16(4n+7)l_{n+2}+64(5n+4)l_{n+1}+512nl_n=0
\end{equation*}
with initial data
\begin{equation*}
l_0=0,\ l_1=1,\ l_2=8,\ l_3=-2,\ l_4=-32.
\end{equation*}
while  the congruences of Theorem 1 amount, via Fermat-Euler's congruence (or directly from Proposition \ref{P.congr}), again to  $l_{kp^{r+1}}\equiv  l_{kp^r}\pmod{p^{r+1}}$.

Also, the binomial expansion alluded to above readily shows that the $l_n$ are integers. To find an explicit formula, we may indeed expand  $y(x)^{-1}=(\sqrt{Q(x)})^{-1}=(\sqrt{4+(x+x^2)^2})^{-1}$ at $x=0$, as the series  ${1\over 2}(1+({x(1+x)\over 2})^2)^{-{1\over 2}}={1\over 2}\sum_{m=0}^\infty{-{1\over 2}\choose m}4^{-m}x^{2m}(1+x)^{2m}$ and, on expanding the terms $(1+x)^{2m}$ again with the binomial theorem and using \eqref{E.s(x)}, we eventually find
$l_n=b_{n-1}+8b_{n-2}$  for $n\ge 2$, where 
\begin{equation*}
b_n=\sum_{r\equiv n\pmod 2}(-1)^{n-r\over 2}4^r{n-r\choose {n-r\over 2}}{n-r\choose r}.
\end{equation*}
We have for instance $b_0=1$, $b_1=4$, $b_2=-2$, $b_3=-16$, $b_4=-26$. This formula, explicit as it is,  however seems not to be very helpful for proving directly the said congruences. 

\smallskip

Another formula, this time  modulo $p$, can be derived using $s(x)={2x(2x+1)\over y(x)}=2x(2x+1)Q(x)^{p-1\over 2}y(x)^{-p}$. This exhibits the polynomial solution $2x(2x+1)Q(x)^{p-1\over 2}$ of (\ref{E.diff.eq.}), considered modulo $p$.

Observe that the  factor $y(x)^{-p}$, considered modulo $p$, does not involve powers $x^m$ with $0<m<p$, hence $s(x)\equiv x(2x+1)Q(x)^{p-1\over 2}+O(x^{p+1})\pmod p$, and, on expanding $Q(x)^{p-1\over 2}$ with the binomial theorem similarly to the above, leads to a certain congruence which we leave to the interested reader to write down. This procedure, like the former, seems not to lead directly even to the special case of the congruence   $c_p\equiv 1\pmod p$.

\smallskip

The above explicit formula leads easily to congruences for the integers $l_n$ modulo powers of $2$. For example, noting  that ${2k\choose k}$ is even for all integers $k>0$, we find that $b_n$ is a multiple of $4$ for odd $n>1$ whereas $b_{2m}\equiv (-1)^m{2m\choose m}\pmod 8$. This leads to $l_{2m}\equiv 0\pmod 4$ whereas $l_{2m+1}\equiv (-1)^m{2m\choose m}\pmod 8$, and in turn we deduce that $l_{2^k+1}$ is  even but not a multiple of $4$; for instance, $\ l_5=-154$. As to the precise power of $2$ dividing $l_n$, one can prove for instance that for odd $n=2m+1$, it is the same as the $2$-adic order of ${2m\choose m}$; this is known to be the sum of the binary digits of $m$, so 
not very simple to describe, e.g. not periodical; for $n\equiv 2\pmod 4$ there is a similar interpretation, but we have not further investigated the exact behaviour  for  $n\equiv 0\pmod 4$, which seems more complicated. 

We also remark that, of course,  if our data would be just slightly different, we would not have any expression so explicit; think for instance of examples where $Q(x)=4+(x+x^2)^2$ is replaced by a `general' polynomial of degree $4$ (with corresponding definitions for $s(x)$). Nevertheless, many of the assertions we have met so far (in particular, the above congruences) would continue to hold.

We finally observe that 
the formula $s(x)=2x(2x+1)/\sqrt{Q(x)}$ may be used to obtain an  `explicit' asymptotic behaviour of the $c_n$:  
one looks at the zeros of $Q(x)$ and uses Cauchy's formula to express $c_n$ with a suitable integral, approximating $1/\sqrt{Q(x)}$ with a sum of terms of the shape $\beta/\sqrt{1-\alpha x}$. From this kind of analysis, which we leave to the interested reader,  it follows in particular that $c_n$ behaves asymptotically like ${1\over \sqrt n}r^n(a\cos(n\theta)+b\sin(n\theta))+O({1\over n\sqrt n}r^n)$, for suitable real numbers $r>0$, $a,b,\theta$, where $\theta/\pi$ is irrational and   $a^2+b^2>0$ \footnote{ The involved numbers can also be made more explicit; for instance, $r^{-1}$ turns out to be the  minimum modulus of the complex roots of $Q(x)$, i.e. $2^{-1}\sqrt{1+\sqrt{65}-\sqrt{2+2\sqrt{65}}}$.}; it also follows that $c_n$ changes sign infinitely often. On the other hand, as generally happens for the coefficients of an algebraic function, it does not seem easy to describe the set of integers $n>0$ such that $c_n=0$;  the  asymptotic just given does not itself yield finiteness, only sparseness, for  the values $a\cos(n\theta)+b\sin(n\theta)$ happen occasionally to be small.\footnote{ It is tempting to believe the conjecture that would describe the set   $\{n:c_n=0\}$ as a finite union of arithmetical progressions plus a finite set (as the Skolem-Mahler-Lech theorem asserts for rational functions). In the present case this would imply finiteness, in view of the said asymptotic.} 

\section{Changing the initial data and a proof of Theorem \ref{T.converse} (i)}

We shall now go ahead by studying the generating functions of solutions of the same recurrence (\ref{E.recurrence}), but with general initial (rational) data. This shall lead us to the proof of Theorem \ref{T.converse}.

\medskip

Let then $S(x)=\sum_{n=0}^\infty C_nx^n\in\Q[[x]]$ be the generating function associated to another solution of the recurrence in rationals $C_n$.  We assume from now on that $S(x)$ is not a linear combination  of $1,s(x)$ with constant coefficients, which  is plainly sufficient to prove the theorem: indeed, the vector $(C_1,C_2,C_3,C_4)$ is proportional to $(1,2,-1/8,-1/2)=(c_1,c_2,c_3,c_4)$ if and only if $S(x)=C_0+C_1s(x)$, and also the value $6C_4+C_2+C_1$ vanishes when $S(x)$ is of this  shape.

Replacing $s(x)$ with $S(x)$  in the calculations performed above,  we find that

\begin{align*}(4+8x+x^2+4x^3+5x^4+2x^5)xS'(x)-(4+16x+x^3+x^4)S(x)= & \\
 -4C_0-16C_0x+(4C_2-8C_1)x^2+(8C_3+C_1-C_0)x^3+ (12C_4+8C_3+2C_2+3C_1-C_0)x^4.\\
 \end{align*}
 We then define 
 \begin{equation*}
R(x):=-4C_0-16C_0x+(4C_2-8C_1)x^2+(8C_3+C_1-C_0)x^3+ (12C_4+8C_3+2C_2+3C_1-C_0)x^4.
\end{equation*}

 This is a polynomial of degree $\le 4$. Clearly, it  is a constant multiple of  $4+16x+x^3+x^4$  if and only if $S(x)=C_0+C_1s(x)$, which we are supposing not to be the case. The above may be written as   the (linear inhomogeneous) differential equation \footnote{ This is sometimes called a `Risch equation': see \cite{[vdP]}.}
 
\begin{equation}\label{E.inh}
A(x)S'(x)-B(x)S(x)=R(x),
\end{equation}
where
\begin{equation*}
A(x)=x(1+2x)Q(x)=x(1+2x)(4+x^2+2x^3+x^4),\qquad B(x)=(4+16x+x^3+x^4).
\end{equation*}

After application of the differential operator ${\d\over\d x}- {R'(x)\over R(x)}$ on the left, this yields a linear (homogeneous) differential equation of order $2$ satisfied by $S(x)$. In the usual `differential' terminology, for which we refer to [DGS], this has $0$ as a regular singularity with rational exponents. \footnote{ Since it shall not be   relevant  for the sequel, we do not pursue here neither in recalling the relevant definition nor in giving the (very easy) details of this.}

A relevant observation now is in the following

\medskip

\begin{lem}\label{L.algebraic}
The series  $S(x)$ does not represent an algebraic function.
\end{lem}


\begin{proof} Suppose the contrary, and let $G$ be the Galois group $\Gal(L/\C(x))$ of a finite normal extension $L$ of $\C(x)$ containing $S(x)$. Conjugating the differential equation (\ref{E.inh}) and subtracting we obtain
 that for each $g\in G$ there exists a constant $c_g\in\C$ such that
 \begin{equation*}
 S^g-S=c_gs.
\end{equation*}
 We have then the cocycle equation $c_{gh}s=S^{gh}-S=(S^h-S)^g+(S^g-S)=c_hs^g+c_gs=(c_h\epsilon_g+c_g)s$, where $\epsilon_g\in\{\pm 1\}$ is determined by $s^g=\epsilon_g s$ (indeed, $s^2\in\Q(x)$). Hence $c_{gh}=c_h\epsilon_g+c_g$. As usual, we sum over $h\in G$ and, setting $c:={1\over |G|}\sum_{h\in G}c_h$,  we obtain that 
 \begin{equation*}
 c=c_g+c\epsilon_g.
 \end{equation*}
Putting now, $S_1:=S+cs$, we have $S_1^g=S^g+cs^g=(S+c_gs)+c\epsilon_g s=S+cs=S_1$. Hence $S_1$ is fixed by $G$ and thus lies in $\C(x)$.\footnote{Here we could also have taken the trace over $\C(x)$ to obtain an equivalent conclusion.}

Let the expansion at $x=\infty$ of $S_1(x)$ be $S_1(x)=ax^d+O(x^{d-1})$; then we have $AS_1'-BS_1=2adx^{d+5}+O(x^{d+4})$. But $AS_1'-BS_1=R$, hence $d\le 0$.  

Then, clearly, $S_1$ may have   poles only at some zero of $A$. If $\alpha$ is the pole and $S_1$ has a local expansion at $\alpha$ starting with $a(x-\alpha)^{-d}$, for $a\neq 0$ and an integer $d>0$, then the expansion at $\alpha$ of $AS_1'-BS_1$ starts with $a(-dA'(\alpha)-B(\alpha))(x-\alpha)^{-d}$. But then we must have $dA'(\alpha)+B(\alpha)=0$. Let us check whether this may happen. If $\alpha=0$, then $A'(\alpha)=B(\alpha)=4$, which is impossible. If $\alpha=-1/2$, then $A'(\alpha)=B'(\alpha)=-65/16\neq 0$, again an impossible event. If $Q(\alpha)=0$, then $A'(\alpha)=2\alpha(1+2\alpha)(PP')(\alpha)$ (where we have put $P(x)=x+x^2$, so $Q=4+P^2$). But then $2dx(1+2x)PP'+4(1+2x)+x^2P$ must be divisible by $Q(x)$, since $Q(x)$ is irreducible over $\Q$; hence $2dx(1+2x)PP'+4(1+2x)+x^2P=l(x)Q(x)$ for some linear $l(x)$. Reducing modulo $P$ we find that $l(x)=1+2x$, which would imply that $l(x)$ divides $xP$, again  impossible. 

Summing up, $S_1(x)$ would have no poles, so would be constant, a contradiction with our assumption that $S(x)$ is not of the shape $C_0+C_1s(x)$. 
\end{proof}

\medskip

\begin{rem} {\bf A convention.} \label{R.convention} Before going ahead, we note that for all of our assertions it makes no difference to replace $S(x)$ by $S(x)-C_0$, which amounts to suppose $C_0=0$ as first datum of our recurrence; note that the choice of this first value does not affect the subsequent values. With this convenient choice, we have $S(0)=0$ and  the equation (\ref{E.inh}) remains true with $R(x)$ replaced by $x^2\tilde R(x)$, where
 \begin{equation*}
\tilde R(x):=(4C_2-8C_1)+(8C_3+C_1)x+ (12C_4+8C_3+2C_2+3C_1)x^2.
\end{equation*}

\end{rem}

\medskip

\subsection{Solving by `quadratures'.} To go ahead, we now solve `by quadratures' the inhomogeneous differential equation (\ref{E.inh}) satisfied by $S(x)$. We put
\begin{equation*}
f(x):={S(x)\over s(x)}.
\end{equation*}
This is a Laurent series in $\Q((x))$; actually, by our convention $C_0=0$, this lies in $\Q[[x]]$. 
Differentiating the equation $S(x)=f(x)s(x)$ we have $S'=f's+fs'$, whence $AS'-BS=f(As'-Bs)+Af's=Af's$. Therefore we find $f'=R/As$ and, for the differential $\d f=f'\d x$,
\begin{equation}\label{E.derivative}
\d f(x)=f'(x)\d x={R(x)\d x\over A(x)s(x)}={\tilde R(x)\over 2(1+2x)^2Q(x)}y(x)\d x={\tilde R(x)\over 2(1+2x)^2}\omega,
\end{equation}
where $\omega:=\d x/y$ is the differential form of the first kind on $E$, which we have already introduced. Observe that the right hand side represents a differential form  relative to the elliptic field $\Q(x,y)$, whereas on the left we have an exact differential of a Laurent series in $\Q((x))$ (in which $\Q(x,y)$ is embedded).

\bigskip

\begin{proof}[Proof of   Theorem \ref{T.converse} (i).] 
We can now prove the first assertion (i) of Theorem \ref{T.converse}. Let us write $f(x)=
f_0+f_1x+\ldots$. Note that by (\ref{E.derivative}) the derivative of $f(x)$ is algebraic; hence by Eisenstein's theorem (see \cite{[DvdP]}) the $n$-th Taylor coefficient of this  derivative  has a  denominator dividing $d^n$ for a suitable integer $d\neq 0$. Actually, a direct calculation (similar and simpler than the one in \S  \ref{SS.explicit}) shows that, since $Q(0)=4$, such a denominator divides $d4^{2n}$, where $d\neq 0$ is a common denominator for the coefficients of $R(x)$.

Hence $d4^{2n}nf_n$ is an integer for all $n\ge 0$; consequently, if $l(n)$ denotes the least common multiple of $1,2,\ldots ,n$, we see that $d4^{2r}l(n)f_r$ is an integer for each $r\le n$. 

Now, $S(x)=f(x)s(x)$, whence  $C_n=
f_0c_n+f_1c_{n-1}+\ldots +f_{n-1}c_1$. Since $2^{2n-2}c_n$ is an integer (as remarked above), we see that $d4^{2n}l(n-1)C_n$ is an integer for all $n$, proving the sought assertion.

\end{proof}


\begin{rem} \label{R.G-f}
(i) The same method outlined in 
\ref{SS.explicit} allows to obtain  asymptotic expressions  for the coefficients $f_n$ and $C_n$, similar to the one written down above for the $c_n$.

(ii) Denominators of coefficients of power-series solutions of linear differential equations over $\Q(x)$ appear in the very definition of $G$-functions: see for instance \cite{[DGS]}, p. 264; for example, the present Theorem \ref{T.converse}(i) implies that $S(x)$ is of this type. For $G$-functions and their arithmetic properties, a vast theory has been developed. For some aspects related to the present context, see for instance \cite{[DGS]}, especially Chapters VII, VIII, and \cite{[A]}.
\end{rem}

\medskip

\section{A question of P\`olya and its evolutions}\label{SS.polya}

In this section we shall discuss a question of P\'olya, strictly related to our issues, especially with Theorem \ref{T.converse}(ii), and with Theorem \ref{T.modp}, a part of which shall be proved. 

\medskip

\noindent{\bf Integrals of algebraic functions and a question of P\`olya}  It is a result of 
G. P\`olya 
that {\it if  a power series representing a rational function has an  indefinite integral having  rational integer coefficients, then  this integral is also a rational function}; see \cite{[P]}, Satz I.

We pause to discuss briefly this simple but interesting issue. Since the indefinite integral of a rational function is anyway the sum of a rational function and a linear combination of   `logarithmic terms', the above  assertion readily reduces to the following one:  

{\it  for  algebraic numbers $\alpha_i,\beta_i$ a function of the shape $\sum_{i=1}^h\alpha_i\log (1-\beta_ix)$  cannot have all its Taylor coefficients (algebraic) integers, 
unless it vanishes.}  

In turn, this amounts to show that  if $h\ge 1$, $\alpha_i\neq 0$ and the $\beta_i$ are nonzero and pairwise distinct, then $\sum_{i=1}^h\alpha_i\beta_i^n$ cannot be always divisible by $n$.\footnote{ We refer to the ring of algebraic integers or to a ring of $S$-integers of an appropriate number field.} We sketch a simple argument to prove this last assertion:  let us assume the contrary, setting $n=p,2p,\ldots ,hp$, for $p$ a large prime number. We obtain that $\sum_{i=1}^h\alpha_i\beta_i^{pm}\equiv 0\pmod p$ for $m=1,2,\ldots ,h$. But then, if $p$ is so large to be coprime with 
the $\alpha_i$, a Vandermonde argument shows that $p$ divides the product of $\beta_i^p-\beta_j^p$ for $i\neq j$; but $\beta_i^p-\beta_j^p\equiv (\beta_i-\beta_j)^p\pmod p$, and then for large $p$  the product of $\beta_i-\beta_j$ would be also divisible by $p$. However this is again impossible  because the $\beta_i$ are pairwise distinct.

\medskip

This  sketch is a  version of  the original argument by P\`olya in \cite{[P]}. Another argument comes on looking directly at the residues of the rational differential in question, and shall be given in a moment and also later in a more general setting. Before this, we pause to prove the result of Honda and Katz mentioned in the Introduction, which shall be relevant and is of interest in itself. 

\subsection{A result of Honda and Katz} \label{SSS.honda-katz} We let $K$ be a function field of dimension $1$ over a finite field $\F_q$ of characteristic $p$.  We remark at once that in the applications below $K$ shall be the reduction modulo $p$ of a function field over $\Q$, which we shall denote by the same letter, or occasionally by $K_p$.

We let $x$ be local parameter at some place of $K$. \footnote{ We tacitly assume that $K/\F_q$ is regular, so $K$ is the function field of a nonsingular curve over $\F_q$; by `place' we refer here to a point of degree $1$ over $\F_q$.} 
After extension of the ground  field if necessary, this amounts to ask that the differential $\d x$ is not zero, or, equivalently, that $K/\F_q(x)$ is separable. Then $K$ is a finite primitive extension of $\F_q(x)$, and may be embedded in $\F_q((x))$ through the said place. 

We set $D:=\d /\d x$; the kernel of $D$ in $\F_q((x))$ is plainly $\F_q((x^p))=(\F_q((x)))^p$, and the kernel of $D$ in $K$ is $K\cap \F_q((x^p))$. As already remarked, this contains $K^p:=\{a^p:a\in K\}$ and thus must coincide with $K^p$ because $[K:K^p]=p=[\F_q((x)):\F_q((x^p))]$. 

 We finally let $\Gamma:=a_0D^n+\ldots +a_{n-1}D+a_n$ be a linear differential operator of order $n$ over $K$, so $a_i\in K$, $a_0\neq 0$. We have the following result, a rephrasing of a small part   of a general theory  due to Honda and Katz, for which we offer an  argument which is apparently different. (See also  \cite{[H]}, \cite{[K]} and \cite{[DGS]}.)

\begin{prop}\label{T.honda-katz} (i) Suppose that $\Gamma(\phi)=0$ has  $m$ solutions in $\F_q((x))$, linearly independent over $\F_q((x^p))$. Then it has  $m$ solutions in $K$, linearly independent over $K^p$.

(ii) Suppose that, for an $u\in K$, $\Gamma(\phi)=u$ has a solution $\phi\in\F_q((x))$. Then it has a solution in $K$.

\end{prop}

\begin{proof} We may write $K=K^p+K^px+\ldots +K^px^{p-1}$, because $1,x,\ldots ,x^{p-1}$ are clearly linearly independent over the constant field $K^p$, and thus a basis for $K/K^p$. These powers of $x$ are also a basis for $\F_q((x))/\F_q((x^p))$, which is the fundamental fact leading to the conclusion.

 Now, if $\phi\in \F_q((x))$, we may express $\phi=b_0^p+b_1^px+\ldots +b_{p-1}^px^{p-1}$ with   $b_i\in\F_q((x))$, and if $\phi\in K$ this expression has the $b_i$ also in $K$. 

Expressing in a similar way the coefficients $a_0,\ldots ,a_n$, it is  very easy to see that the differential equation $\Gamma(\phi)=0$   may be written as a homogeneous linear $p\times p$ system over $K$, in the unknowns $b_0,\ldots ,b_{p-1}$. 
The assumption says that the space of solutions of  this linear system in $\F_q((x))$ has dimension at least $m$ over $\F_q((x))$. But since the system is defined over $K$, the kernel in $K$ has again dimension at least  $m$ over $K$. Now, using the said correspondence with the solutions of the differential equation, this proves the first assertion.

The second one is obtained in the same way.

\end{proof}

\begin{rem} \label{R.honda-katz} (i) Note that in the homogeneous case we may multiply a solution by any element of $\F_q[x^p]$ and still obtain a solution; in particular, we may achieve that the said $m$ independent solutions in $K$ are integral over $\F_q[x]$; so, if $K=\F_q(x)$ we may assume that they are polynomials.

(ii) Of course, the dimension of the space of solutions over the constant field cannot exceed the order $n$ (which can be proved on considering the Wronskian matrix; see \cite{[DGS]}). So if $m=n$ the conclusion guarantees a full system of solutions in $K$. 

(iii)  The proof also shows that any given solution in $\F_q((x))$ of the homogeneous system can be approximated up to any prescribed order with a solution in $K$. See also Lemma 1 of \cite{[H]} for this more precise conclusion.

(iv) Much the same argument yields also existence results for non-linear differential equations; this time, for instance,  a (nonzero) solution in $\F_q((x))$ shall yield a (nonzero) solution in a separably algebraic extension of $\F_q(x)$.  (For the separability, which is quite relevant here, one has just to note in the proof that if an algebraic extension of $k(x)$ may be embedded in $k((x))/k(x)$ then it is separable.)


\end{rem}


\begin{example}\label{E.pol} We have noted  in \S \ref{SS.explicit} that $s(x)=y(x)^{-p}2x(2x+1)Q(x)^{p-1\over 2}$. Hence the polynomial $x(2x+1)Q(x)^{p-1\over 2}\in \F_p[x]$, of degree $2p$, is a nonzero solution of the differential equation (\ref{E.diff.eq.}), considered over $\F_p$ (a fact also immediately checked directly). For $p\neq 2, 5, 13$ this polynomial is easily seen not to have roots of multiplicity divisible by $p$, hence  is a polynomial solution of least degree, thus unique up to a constant factor. (See for this the part of the proof of Theorem \ref{T.modp} at the end of \S \ref{SS.polya}.)

\end{example}

\bigskip

Let us now go back to P\`olya's issue.

We note that the given argument (and also the alluded one concerning residues) proves the  much stronger assertion obtained by   assuming the weaker fact that  
{\it the said Taylor coefficients lie in $\Z_p$  for infinitely many $p$}, rather than in $\Z$.

\medskip

We can go further, and even replace $\Z_p$ with $\F_p$; let us see how.  Denoting $g(x)$ the rational function in question, we are seeking a solution of $f'(x)=g(x)$, which leads to $ f''(x)-{g'(x)\over g(x)}f'(x)=0$. This is a second-order differential equation over $\Q(x)$; it is {\it reducible} and  its non-constant solutions yield nonzero solutions of the former inhomogeneous equation. Reducing the  latter equation   modulo $p$ for large primes $p$, we are directly linked with the Grothendieck conjecture.\footnote{ See \cite{[vdP]} for a discussion of the Grothendieck conjecture for second order operators defined over $\Q(x)$.}

Suppose now that 
there is a nonconstant solution 
 in $\F_p((x))$ of the inhomogeneous  equation $Y'=g$; then, 
 part (ii) of Proposition \ref{T.honda-katz}   implies the existence of a 
 solution $f_p(x)\in \F_p(x)$. 
Note that this amounts to the fact that the coefficients of powers $x^{mp-1}$ in a power-series expansion of $g(x)$ vanish modulo $p$.

 
Then the differential $g(x)\d x$ is exact modulo $p$, i.e. as a differential of $\F_p(x)$, which implies that its residues  (modulo $p$) vanish.  Now, if  $p$ is sufficiently large to ensure that the zeros/poles of $g(x)\d x$ (as a differential in characteristic $0$) do not collapse by reduction,  the  residues of the differential considered modulo $p$ are the reductions (modulo a place above $p$ of a suitable number field) of the residues of $g(x)\d x$ considered as a complex differential. 
 
 If this happens for infinitely many primes, then the  residues  of $g(x)\d x$, considered in characteristic zero, must vanish, so the differential $g(x)\d x$ is exact over $\Q$.

 This argument yields a second proof of P\`olya's result, now showing  that  
 
 \smallskip
 
  {\it If for infinitely primes $p$ (the reduction of) $g(x)$ has a power series indefinite  integral in $\F_p((x))$, then the indefinite integral of $g(x)$ in $\Q((x))$ is a rational function.}. 
  

  \smallskip
  
  Note that the assumption now is indeed {\it a priori} weaker than P\`olya's: this is shown for instance by the following 
  
  \begin{example}\label{Ex.F_p-Z_p}  Let us consider the rational function $g(x)={p\over 1-x}$. Naturally, $g(x)\equiv 0\pmod p$, so it has an indefinite integral in $\F_p((x))$. However, any indefinite integral of $g(x)$ in characteristic $0$ is given, up to the addition of a constant, by the series $\sum_{m=1}^\infty {p\over m}x^{m}$.  Hence the series does not lie in $\Z_p[[x]]$ and $g(x)$ has no indefinite integral in $\Z_p[[x]]$.
  
 This is a rather trivial example, and another one, much more significant, appears in connection with a differential form $\omega$ of the first kind on an elliptic curve, for instance defined over $\Q$: as we shall recall in Example \ref{Ex.elkies} below, there are infinitely many primes for which $\omega$ is exact modulo $p$, but, if $t=x/y$ denotes the standard Weierstrass local parameter at the origin,  $\omega$ does not become exact in $\Z_p[[t]]$. \footnote{ Proposition \ref{T.honda-katz}, or the Cartier operator  introduced below, show that being exact modulo $p$ does not depend on the local parameter used to express the differential form, and, less obviously, does not depend on the chosen  {\it place}. It may be easily shown that being exact in $\Z_p[[x]]$ does not depend on the local parameter provided the local parameters may be expressed each other in terms of series over $\Z_p$. 
 In the Appendix we shall briefly consider in special cases how being exact in $\Z_p[[x]]$ depends on the chosen place.}
  \end{example}
  
\medskip




\medskip

Now, apart from these different shapes of the assumptions, P\`olya suggested that a result   analogous to his own one given above  should hold, more generally,  for {\it algebraic functions} in place of {\it rational functions}. The above argument certainly does not extend in this sense, and indeed the general assertion is extraordinarily more difficult. A general proof was given by Andr\'e, as a corollary of results proved  in his book \cite{[A]} (see Corollary at p. 149) and then again in the paper \cite{[A2]} (see Proposition 6.2.1), where more details are given.  We shall soon state this in more precision; now we note that the alternative  argument,  using residues, that we have given for P\`olya's result in the case of rational functions, proves exactly in the same way the following assertion in the direction of P\`olya's question, where we restrict to the rational number field merely for the sake of simplicity, and where (as always) by {\it exact in $k((x))$} we mean {\it of the shape $\d h:= h'(x)\d x$ for some $h\in k((x))$}: 

\begin{prop}\label{P.residues} Suppose that $K_p$ is a function field of dimension $1$ defined over $\F_p$, with a rational function $x$ which is a local parameter at some place of $K_p$, rational over $\F_p$. Let us embed $K_p$ in $\F_p((x))$ through this place. Assume that  the differential $g\cdot \d x$  
 is exact in $\F_p((x))$, for a $g\in K_p$. 
 Then the differential has no residues. 
 
 In particular, if $K$ is a function field over $\Q$, if $g,x\in K$ and if  $g\cdot \d x$ is exact in the reduction $K_p$ of $K$  for infinitely many primes,  it has no residues as a complex differential. 
\end{prop}

Of course,  if $g\in K\subset \Q((x))$ may be written as $f'(x)$ for a series $f\in\Q((x))$ which lies in $\Z_p[[x]]$, 
then $g\cdot \d x$ is exact in $\F_p((x))$ and the proposition may be applied, proving that it has no residues modulo $p$.

We have avoided any technicalities related to good reduction, leaving the easy details for this case to the interested reader. \footnote{ Other facts on good reduction are by far more subtle; there is of course relevant literature, starting e.g. with Deuring and Shimura, which however we do not even refer to in precision since this falls  out of our present purposes.}

\begin{proof}
A proof follows as in the case of rational functions: if $g(x)\d x$ is exact as a differential in $\F_p((x))$, then it is exact in $K_p$, in view of Proposition \ref{T.honda-katz} (ii); then it has no residues  at any place of $K_p$. (Note that a direct use of the assumption  implies this conclusion merely for the place giving the embedding of $K_p$ in $\F_p((x))$.)

Now  for large primes, so large to ensure that the zeros/poles of the differential do not collapse by reduction, the residues modulo $p$ are obtained just by reduction of the residues in characteristic $0$; and  then the residues of $g \cdot \d x$ vanish modulo $p$ (or better, modulo some prime of  a number field containing the residues) for infinitely many primes, whence they must vanish. 
\end{proof}

\begin{rem}\label{R.residues} (i)  Of course the  conclusion about residues is generally weaker than   in P\`olya's expectation. Indeed, in non-rational function fields of curves, the differentials without residues (i.e., {\it of the second kind}) are not necessarily exact. Actually, with such an assumption we cannot expect P\`olya's conclusion: an example  (already considered above)  shall be pointed out below in no. \ref{Ex.elkies} and concerns the differential form $\omega$ of the first kind on an elliptic curve over $\Q$; in the case of Complex Multiplication, the assumption to be exact modulo $p$ is even verified by a set of primes of positive density. An explicit instance  is given by the curve $y^2=x^3-x$ with the form $\omega=\d x/y$, which is exact modulo all primes $\equiv 3\pmod 4$, but  is certainly not exact in $\C(x,y)$ and, for any such prime,   not even in $\Z_p[[t]]$, where $t=x/y$ is the `standard' local parameter at the origin.

(ii) Naturally, if  the function field $\Q(x,g(x))$ is rational, i.e. has genus zero, then, having no residues amounts to be exact, so indeed Proposition \ref{P.residues}  yields an affirmative answer to P\`olya's question for this special case, even if the function $g(x)$ itself is not rational. On the other hand, this may be derived directly from P\`olya's result, as follows. By assumption, the said field is of the shape $\Q(t)$, and we  we may also assume that the function $t$ vanishes (and is thus a local parameter) at the place corresponding to the embedding of the field in $\Q((x))$. By assumption, the differential $\xi:=g(x)\d x$ is of the shape $\d f(x)$ where $f\in\Z[[x]]$. Now, we have $x=r(t)$ for a certain rational function $r\in\Q(t)$ with $r(0)=0$; hence the formal series $f(r(t))$ yields a series $F(t)\in \Q[[t]]$ by expanding each term $r^n(t)$ in $\Q[[t]]$, this being justified since $r^n(t)$ has order $\ge n$ at $t=0$. In this expansion the coefficients are $p$-integral except at finitely many primes, and hence the same holds for $F(t)$. Now, we have that $F'(t)=g(r(t))r'(t)$ is a rational function of $t$. Then an application of P\`olya's result gives what is required. 

\end{rem}

\medskip

We may now apply Proposition \ref{P.residues} to our main issue, to obtain part of  Theorem \ref{T.modp}. Maintaining the notation for that statement,  we prove:

\begin{prop}\label{P.modp}  The space $W_p$ is infinite dimensional over $\F_p$. For $p\neq 2,3,5,13$, the space $V_p$ is contained in the hyperplane defined by $6\bar C_4+\bar C_2+\bar C_1=0$.

\end{prop}


\begin{proof} 
For $p>2$, let  $\bar s(x)$  denote the reduction of $s(x)$ modulo $p$, which makes sense by Theorem \ref{T.congr}. Observe that $W_p$ contains all sequences arising from the coefficients of a  series of the shape $\lambda(x^p)\bar s(x)$, for any choice of $\lambda\in\F_p[[x]]$. This yields infinite dimension. For $p=2$, the same argument works on using series $\lambda(x^2)(1+x)$.

For the other assertion,  let $(\bar C_n)\in W_p$ 
and put $\bar S(x)=\sum_{n=0}^\infty \bar C_nx^n\in \F_p[[x]]$. \footnote{ We stress again that here and in the whole paper by `$\bar S(x)$' we generally denote a series in $\F_p((x))$, not necessarily meant to be the reduction of a series in characteristic zero verifying the recurrence, unless this is stated explicitly.} We may assume as above that $\bar C_0=0$ and that $\bar S(x)$ is not a constant multiple of   $\bar s(x)$, for otherwise the conclusion would hold, since $6c_4+c_2+c_1=0$. 

We may consider as above the series $\bar f(x)=\bar S(x)/\bar s(x)$, obtaining that  the differential form ${\tilde R(x)\over 2(1+2x)^2}\omega$ on the right of (\ref{E.derivative}) is exact in $\F_p[[x]]$,  
where now $\tilde R(x):=(4\bar C_2-8\bar C_1)+(8\bar C_3+\bar C_1)x+ (12\bar C_4+8\bar C_3+2\bar C_2+3\bar C_1)x^2$.


An application of Proposition \ref{P.residues} yields that it has no residues.  An easy computation shows that the residue at any point with $x=-1/2$ is given by  $\tilde R'(-1/2)$, up to a power of $2$; hence $\tilde R'(-1/2)=0$, which amounts to
$12\bar C_4+8\bar C_3+2\bar C_2+3\bar C_1= 8\bar C_3+\bar C_1$, i,.e.
\begin{equation*}
6\bar C_4+\bar C_2+\bar C_1=0,
\end{equation*}
as required.

\end{proof}


We conclude \S \ref{SS.polya} by an estimation of the minimal index $n$ for which $p$ appears in the denominator of $C_n$, when the relevant differential form is not exact modulo $p$. We have

\begin{prop}\label{P.estimate}
There is a constant $l$ with the following property. Suppose that the reduction of the form ${\tilde R(x)\over 2(1+2x)^2}\omega$ on the right of (\ref{E.derivative}) is not exact in $\F_p[[x]]$. Then there is an $n\le l\cdot p$ such that $C_n$ is not $p$-integral.
\end{prop}

\begin{proof} 
Let $g(x)=f'(x)=\sum_{n=0}^\infty nf_nx^{n-1}$,  so the assumption says that $g(x)\d x$ is not exact modulo $p$.  
Suppose that $C_n$ is $p$-integral for $n\le n_p$. Then the same holds for $f_n$ and, 
 writing $g(x)=\sum_{n=0}^\infty g_nx^n$, we have that    $p|g_m$ for all $m\equiv -1\pmod p$, $m\le n_p-1$. But then the $(p-1)$-th derivative $g^{(p-1)}(x)$ would vanish modulo $p$ at the origin to an order $\ge n_p-p$. 

Now recall from (\ref{E.derivative}) that $g(x)$ is algebraic. Let us then estimate the degree (as rational functions on the appropriate curve) of its derivatives modulo $p$. In general, let $h(x)$ be an algebraic function in a function field $K$ of a curve,   let  $v$ be a place of $K$, and $t=t_v$ be a local parameter at $v$. Then $v(h'(x))=v(\d h/\d t)-v(\d x/\d t)\ge v(h)-1-v(\d x/\d t)$, where we are using the same letter $v$ to denote the order  function $v(\cdot )$ at the place $v$. 

Now, the sum over $v$ of the positive ones among the $v(\d x/\d t)$ is bounded; 
such a   bound remains true if we consider the reduction modulo $p$, provided $p$ is large enough: indeed, the value $v(\d x/\d t)$ is zero except at the branching points of the function $x$, and since there are only  finitely many of them,  the value of the sum in zero characteristic may be increased by reduction only modulo finitely many primes. In conclusion, since the degree of a function is the number of its poles counted with multiplicity, we have   $\deg (h'(x))\le \deg (h)+m_h+l_2$, where $m_h$ is the number of poles of $h$ and  for a constant $l_2$, even if we reduce the function $h$ modulo $p$, provided $p$ is larger than a constant depending only on $K$ and $x$. Iterating this, and taking into account that the derivative (with respect to $x$) has at most $m_h+m_x$ poles, we find that $\deg(h^{(r)})\le\deg (h)+r(m_h+m_x+l_2)\le (r+1)(\deg(h)+l_3)$, where $l_3$ does not depend on $p$ or $h$ (only on $K,x$). Therefore $\deg (g^{(p-1)})\le p l_4$ for an $l_4$ independent of $p$. 

On the other hand, either $g^{(p-1)}$ vanishes modulo $p$, or its degree is $\ge n_p-p$, since it has a zero of that order at the relevant place where $x=0$ (i.e. the place corresponding to the point $(0,2)$ on $E$).  In the first case the differential form $g\cdot \d x$ is exact modulo $p$. In the second case, $n_p\le (l_4+1) p$, as required. (For the exceptional finitely many primes for which the above argument may fail, such an estimate also holds at the cost of increasing $l$.)

\end{proof}

\begin{rem}\label{R.estimate} We note that  the conclusion in this proposition holds for a general differential form $\xi$ of a function field of a curve over $\Q$ or a number field (with a constant $l$ dependent on $\xi$ in  a way that can be estimated from the given proof). A more explicit estimate shall be given with another method based on the Cartier operator: see Proposition \ref{P.estimate2} below.

Also, we note that the conclusion is not generally true if we merely assume that 
$\xi$ is not exact in $\Z_p[[x]]$. Indeed, as we have seen in Example \ref{Ex.F_p-Z_p}, the condition that the reduction of $\xi$ is exact in $\F_p[[x]]$ is generally weaker than being exact in $\Z_p[[x]]$, and the second of the examples given there, together with the Atkin \& Swinnerton-Dyer congruences that shall be recalled later,  shows that for infinitely many primes the first occurrence of $p$ in a denominator of an indefinite integral corresponds to a term $x^{p^2}$.

We also remark that this result  shall be relevant for Proposition  \ref{C.lowerbound}, proved in the next section.

\end{rem}

\begin{proof}[Continuation of the proof of Theorem \ref{T.modp}]  We go ahead with this theorem, by proving that the space $V_p$ has dimension $2$ or $3$ and that it contains linearly independent vectors given by $(1,2,-1/8,-1/2)$ and $(\bar c_{p+1},\bar c_{p+2},\bar c_{p+3},\bar c_{p+4})$.

For this, we consider the reduction $\bar s(x)\in \F_p[[x]]$ of our series $s(x)$. We write
\begin{equation*}
\bar s(x)=x\sigma_0(x)+x^{p+1}\sigma_1(x)+\ldots +x^{mp+1}\sigma_m(x)+\ldots ,
\end{equation*}
where $\sigma_0,\sigma_1, \ldots$ are polynomials of degree $\le p-1$ over $\F_p$; specifically, $\sigma_m(x)=\bar c_{pm+1}+\bar c_{pm+2}x+\ldots +\bar c_{pm+p-1}x^{p-1}$. 

We also let $\Gamma=A{\d \over \d x} -B$ denote the differential operator as in equation \eqref{E.inh}, so $\Gamma(\bar s(x))=0$. Setting for this proof $r_m(x):=x\sigma_m(x)+x^{p+1}\sigma_{m+1}(x)+\ldots $, we    have  $0=\Gamma(x\sigma_0(x)+\ldots +x^{(m-1)p+1}\sigma_{m-1}(x))+x^{mp}\Gamma(r_m(x))$. The first term on the right is a polynomial of degree at most $mp+4$ \footnote{ Observe in fact that a possible term of degree $pm+5$ cannot appear because of differentiation.}, hence $\Gamma(r_m(x)))$ is a polynomial of degree $\le 4$. This proves that the sequence of coefficients of $r_m(x)$, i.e., the sequence $(\bar c_{mp+n})_{n\in\N}$, satisfies the recurrence \eqref{E.recurrence}, for each integer $m\ge 0$, and thus lies in $W_p$. \footnote{ Note in passing that, in view of Proposition \ref{P.modp}, this implies the congruence $6\bar c_{mp+4}+\bar c_{mp+2}+\bar c_{mp+1}=0$ for each $m\in\N$, mentioned in the introduction. A direct proof of such congruence shall however be given in \S \ref{SS.cartier} below.} 

We now prove that, by projection  on the first four coordinates, the elements $\bar s(x)=r_0(x)$ and $r_1(x)$ give rise to linearly independent elements of $V_p$. Indeed, suppose the contrary. Then, since the terms $\bar C_1,\ldots ,\bar C_4$ determine the terms up to $\bar C_p$ for any solution $(\bar C_n)$ of \eqref{E.recurrence}, we would obtain that $\sigma_0(x),\sigma_1(x)$ are linearly dependent over $\F_p$, hence $\sigma_1(x)=\lambda \sigma_0(x)$ for a $\lambda\in\F_p$.  

This yields $\bar s(x)=(1+\lambda x^p)x\sigma_0(x)+O(x^{2p})$. Applying the operator $\Gamma$ and taking into account that differentiation in $\F_p((x))$ is $\F_p((x^p))$-linear, we obtain $(1+\lambda x^p)\Gamma(x\sigma_0(x))=O(x^{2p})$, whence on multiplying by $(1+\lambda x^p)^{-1}=1-\lambda x^p+\lambda^2x^{2p}-\ldots$, we get $\Gamma(x\sigma_0(x))=O(x^{2p})$. However the left side is a polynomial of degree at most $p+4$, and therefore, since $p>3$, it must vanish. This means that $x\sigma_0(x)$ is a polynomial solution of \eqref{E.diff.eq.}, of degree $\le p$.

Now recall from Example \ref{E.pol} that the polynomial $x(2x+1)Q(x)^{p-1\over 2}\in \F_p[x]$, of degree $2p$, is a nonzero solution of the differential equation (\ref{E.diff.eq.}), considered over $\F_p$. But such a differential equation is of the first order and thus any two solutions over the field $\F_p(x)$ are linearly dependent over the constant field, which in this case is  $\F_p(x^p)=\F_p(x)^p$. Therefore we would have $a(x)^p\sigma_0(x)=b(x)^p(2x+1)Q(x)^{p-1\over 2}$, for suitable polynomials $a,b\in\F_p(x)$, not both zero, and which can be therefore supposed to be coprime. In particular, $a(x)^p$ would divide $(2x+1)Q(x)^{p-1\over 2}$; but $Q(x)$ has no repeated factors over $\F_p$ for $p\neq 2,5,13$, hence $a(x)$ would be constant. But this contradicts that $\deg\sigma_0\le p-1$, $\deg Q=4$. 

In conclusion, the said vectors are linearly independent over $\F_p$ and in particular  yield that $\dim V_p\ge 2$. On the other hand, by Proposition \ref{P.modp}, $V_p$ has dimension at most $3$.

\end{proof}

 In \S \ref{SS.cartier} we shall finally complete the proof of Theorem \ref{T.modp} by proving the last assertion left open, i.e. that indeed $\dim V_p=2$.



\section{A theorem of Chudnovski and Andr\'e and a 
proof 
of Theorem \ref{T.converse} (ii)} \label{S.andre}

We shall now discuss briefly P\`olya's question about a possible extension of his result to the case of algebraic (non rational) functions; that is, we ask whether the integrality of the coefficients in an indefinite integral $f(x)$ of the algebraic function $g\in\Q[[x]]$ implies that $f(x)$ itself is algebraic. As in Remark \ref{R.residues}, Proposition \ref{P.residues} yields a weaker conclusion than being exact (although under a weaker assumption). 

To our knowledge, the first complete proof of P\`olya's expectation was given by D.V. and G.V.  Chudnovski \cite{CC} and then by Andr\'e, and stated in his book \cite{[A]}, p. 149, where a proof can also be found; this is further  detailed in \cite{[A2]}.  However a substantial  part of the method and result appeared already in \cite{CC}.
Let us state in precision this result  in a phrasing slightly different than in \cite{CC} and \cite{[A2]} and working for simplicity over $\Q$ in place of a general number field:

\medskip

\noindent{\bf Theorem of Chudnovski and Andr\'e's theorem} (\cite{CC}, \cite{[A2]}, Prop. 6.2.1.) {\it  Let  $\xi $ be  an algebraic differential of a function field $K$ of  a curve defined over $\Q$. Suppose  that  for all primes $p$ except possibly for a set of  zero density \footnote{ By this we mean, as usual,   that the number of these primes up to $T$ is $o(T/\log T)$.} it is exact when considered modulo $p$. Then $\xi$ is exact as a  differential of $K$.}


\medskip

First of all, let us note that this implies at once the following assertion, on expressing the differential as $g(x)\d x$ for an algebraic function $g$ represented by a power series, and letting  $f(x)$ be an indefinite integral of $g(x)$. 

\medskip

\noindent {\bf Corollary.} {\it If $f(x)$ is a power series with coefficients in a number field, which are $p$-integers for all primes $p$ except possibly for a set of primes of zero density, and if $f'(x)$ represents an algebraic function, then $f(x)$ represents an algebraic function as well.}

\medskip

\underline{\it About the proof of the Chudnovski-Andr\'e's theorem.}   The arguments of the Chudnovski and Andr\'e's proofs  are delicate and combine various tools;  here we only say a few words on them, thinking for instance of the last statement.   This is deduced as a 
corollary of a general   `algebraicity criterion' for power series in several variables over a number field: see \S 5 in \cite{[A2]}, especially Thm. 5.4.3 and Cor. 5.4.5.  The main assumptions  in this Theorem 5.4.3 are of `$G$-functions' type and concern (a)  the $v$-adic growth of the coefficients and (b) the  $v$-adic radii of convergence of the relevant function $f(x)$: the product of these radii  is assumed to be $>1$. A very  important issue is that it suffices that this last inequality  holds possibly after a preliminary {\it uniformization}  of the relevant functions, 
depending on the place $v$.   

Now, to apply this general criterion to our context, one uses  the integrality assumptions for the coefficients to imply   the first half (a) of the assumptions. As to (b),  the fact that the derivative is algebraic, itself implies that the $v$-adic radii are $\ge 1$ except at finitely many places, where they are positive. 
Instead, for infinite places, 
one uses a crucial uniformization by entire functions of several variables, obtained through the exponential map on a generalized Jacobian.  This yields infinite radius of convergence, which makes automatic the needed inequality. \footnote{ This kind of idea had been introduced and used already in \cite{CC} by D. and G. Chudnovski.} 
 It is worth mentioning that in this step one has to go to the case of several variables even if the original functions depend on a single variable.\footnote{It is tempting to note the analogy with the proofs of Weil, Siegel and Faltings, for integral and rational points on algebraic curves, when one has to study the Jacobian of the curve. This principle here appears to us to be more hidden {\it a priori}.}

The technical part in the proof of the said criterion is through  auxiliary constructions which are familiar  in transcendental number theory. One first  constructs a polynomial in the variables and in the relevant function, having algebraic coefficients with  small  height and vanishing to a high order at the origin; this order and the degree of the polynomial are parameters which eventually are taken to be large enough.  It is then shown that the polynomial is zero (whence algebraicity follows) by comparing estimates for the first nonvanishing coefficient at the origin: an upper bound follows from Cauchy's theorem, and exploits the estimates mentioned above; a lower bound follows from the product formula. Under the said assumptions, these bounds are eventually inconsistent, proving that no nonzero coefficient exists. 

We stress that the above statement is merely one of the many corollaries that can be derived from  this  general setting.

\medskip

As to our situation, these results  
are highly relevant. 
Indeed, we may deduce Theorem \ref{T.converse}(ii) from them.   



\begin{proof}[Proof of Theorem \ref{T.converse} (ii)]  Assumptions as in  the theorem, suppose by contradiction that the said sequence $\{C_n\}$   consists  of $p$-integers at all primes $p$ but  a set of zero density. 
Then $S(x)$  would have $p$-integral coefficients for these primes $p$, and the same would hold for $f(x)=S(x)/s(x)$. But, in view of (\ref{E.derivative}), the derivative of $f(x)$ is algebraic. Hence by the Corollary of Andr\'e's theorem $f(x)$ would be also algebraic. But then $S(x)$ would be algebraic as well, against Lemma \ref{L.algebraic}.
\end{proof}




Using the  above version of Andr\'e's theorem and Proposition \ref{P.estimate}, we can obtain the more precise result of  Proposition \ref{C.lowerbound}.

\begin{proof}[Proof of Proposition \ref{C.lowerbound}]  Let $\mathcal P$ be the set of primes such that the differential form $\xi$   appearing on the right hand side of \eqref{E.derivative} is not exact modulo $p$. Under the present assumptions, $S(x)$ is transcendental, hence $\xi$ is not exact as a differential form on our curve $E$ whence, by Andr\'e's Theorem, the set $\mathcal P$ has not zero density. Therefore there are  a number $b>1$ and  infinitely many integers $m$ such that the product of primes in $\mathcal P$ up to $m$ is at least $b^{m}$. By Proposition \ref{P.estimate},  $l$ being  the constant which appears therein, for each prime $p\in\mathcal P$  there is an index $n\le lp$ such that $C_n$ is not $p$-integral. Hence, letting $n=n_m$ be the integral part  of $lm$,  the product of primes appearing in a common denominator for $C_1,\ldots ,C_{n}$ is at least $b^{m}\ge (b^{1\over l})^{n}$. This proves the assertion, with $a:=b^{1\over l}$.

\end{proof}

This argument also allows to prove Corollary \ref{C.converse}:

\begin{proof}[Proof of Corollary \ref{C.converse}] Suppose $6C_4+C_2+C_1\neq 0$. Then, in view of Proposition \ref{P.modp} (i.e. the part of Theorem \ref{T.modp} proved so far), we see that  for large enough primes $p$  the reduction modulo $p$ of $(C_n)$ cannot lie in $W_p$, a contradiction which proves the first part of the corollary.

As to the second part, an  argument quite similar to that just given for Proposition \ref{C.lowerbound}, now with $\mathcal P$ containing all but finitely many primes,  yields the stated  estimate for all large enough integers $n$.

\end{proof}

\section{The Cartier operator}\label{SS.cartier} For the rest of the arguments it shall be very helpful to dispose of  the so-called `Cartier operator'; we shall now recall the definitions and principal properties, following  Appendix 2 of \cite{[L]}; we shall not reproduce the elegant  proofs given therein, which are however rather self-contained and neither long nor  complicated. 

\medskip

For simplicity we do not operate in maximal generality and restrict to what is sufficient for our context. Let $k_0$ be a perfect field of characteristic $p>0$, and let $L/k_0$ be a finitely generated extension such that the field $L^p:\{u^p:u\in L\}$ is such that $[L:L^p]=p$. In the applications below $L$ shall be either a power series field $k_0((x))$ or a function field of dimension $1$ over $k_0$. Indeed,  it is part of the  standard elementary theory (see e.g. \cite{[L3]}) that in this case there is a {\it separating} element $t\in L$ such that $[L:k_0(t)]$ is separable, and the needed condition $[L:L^p]=p$ follows.

Let  $x\in L$ be any separating element; then $x$ is necessarily purely inseparable of degree $p$ over $L^p$, and $L=L^p(x)$. Let $\Omega=\Omega_L$ be the space of $1$-differential forms on $L/k_0$, and let $\xi\in\Omega$; then $\xi$  may be written as $\xi =u\ d x$, with $u\in L$. In turn, we may write uniquely $u=u_0^p+u_1^px+\ldots +u_{p-1}^px^{p-1}$, with $u_0,\ldots ,u_{p-1}\in L$. We set
\begin{equation*}
C(\xi):=u_{p-1}\d x.
\end{equation*}

This operation, which sends the space $\Omega$  to itself,  is clearly $k_0$-linear, once $x$ is fixed, and we call $C$ the {\it Cartier operator} on $K$. Clearly, it satisfies $C(z^p\xi)=zC(\xi)$ for each $z\in L$.

A most important  and nonobvious fact is that $C$ is well-defined independently of $x$, namely that its value is independent of the separating element $x$ used to express $\xi$. This deduction is not very complicated, but it relies on a number of steps and other  definitions, so we omit a proof and refer to \cite{[L]}, Appendix 2, especially \S 3, \S 4. 

The value $C(\xi)$ picks in particular the  residues in positive characteristic, but expresses more than that; indeed, note that for a differential $u(x)\cdot \d x$ to be integrated, where $u(x)\in k_0((x))$, we do not merely need that its residue vanishes, but that all terms $x^{mp-1}$ in have zero coefficient, i.e. we need that $C(u\d x)=0$. 

Note that if $L=k_0(\X)$ is a function field, corresponding to a smooth complete curve $\X/k_0$,  and if $t$ is a local parameter at some point  $Q\in\X(k_0$ (so in particular $L/k_0(x)$ is separable), we may embed $L$ (through $Q$) in $k_0((t))$ and it may be easily seen that  the Cartier operator then coincides with the one defined for the power series field.

It is clear that the Cartier operator commutes with automorphisms of $L/k_0$. (Indeed, any such automorphism leaves $L^p$ fixed.)

Here are other properties (proved in \cite{[L]}) that we shall use below:

\medskip

(C1): We have $C(\xi)=0$ if and only if $\xi$ is exact, i.e. of the shape $\xi=\d \phi$ for a $\phi\in L$.

(C2): We have $C(\xi)=\xi$ if and only if $\xi$ is {\it logarithmically exact}, i.e. of the shape $\xi=\d\phi/\phi$ for a $\phi\in L^*$.

\medskip

\begin{example}\label{E.modp} Let us illustrate a couple of applications of these properties.

(i) It is a simple calculation to check that $2{x^3+x^2+6\over x^4}t\omega=\d({y\over x^3}+3{y\over x^2})$, where as above $y^2=Q(x)$ defines our elliptic curve and $\omega=\d x/y$. Then $C(2{x^3+x^2+6\over x^4}t\omega)=0$, which yields the congruences $6c_{kp+4}+c_{kp+2}+c_{kp+1}$ mentioned in the introduction. Many similar ones may be derived by an analogous observation. \footnote{ Of course this amounts to the fact that the coefficients of $x^{np-1}$ in a derivative of a function carry a factor $np$.} 

(ii) We offer another proof, very short,  of the case $r=0$ of Proposition \ref{P.congr}. 
 Let now $\xi=x\d f/f$, where $f\in k_0[[x]]$ has $f(0)=1$. 
 Then $C(x^{p-1}\xi)=xC(\d f/f)$. Suppose $f$ is a separating element; then, on viewing $f$ as a new variable (that is, $f$ in place of $x$), we see that $C(\d f/f)=\d f/f$. This shows that $C(\sum c_nx^{n+p-1}\d x)=\xi$, whence $c_{np}=c_n$, as required. (If $f$ is not separating, then $\xi=0$.)
 
 \medskip

\end{example}

We now prove a simple result which bounds the order of poles for $C(\xi)$. Let $L=k_0(\X)$, $Q\in \X(k_0)$ and let $v=v_Q$ denote the order function at $Q$. We have

\begin{prop}\label{P.v_Q} For a differential form $\xi\in\Omega_L$ we have 
\begin{equation*}
v_Q(C(\xi))\ge \lceil {{v_Q(\xi)+1\over p}-1}\rceil .
\end{equation*}
In particular, if $\xi$ is regular at $Q$, the same holds for $C(\xi)$, and if $v_Q(\xi)\ge -p$ then $C(\xi)$ has at most a simple pole at $Q$.
\end{prop}

\begin{proof} Let $t\in k_0(\X)$ be a local parameter at $Q$, and write $\xi=(u_0^p+u_1^pt+\ldots +u_{p-1}^pt^{p-1})\d t$, where $u_i\in L$. Note that the values $v_Q(u_j^pt^j)=pv_Q(u_j)+(t-j)$ are pairwise incongruent modulo $p$, hence pairwise distinct. Hence $v_Q(u_{p-1}^pt^{p-1})\ge v_Q(\xi)$. Therefore  $v_Q(u_{p-1})\ge {v_Q(\xi)+1\over p}-1$. But the right side equals $v_Q(C(\xi))$ and is an integer, whence the assertion.

\end{proof}

For instance, this yields the assertion made in the introduction that a congruence $c_{kp+i}\equiv c_{i}\pmod p$ cannot hold for all $k$. (Here $(c_n)$ is our usual sequence.) Indeed, it is a very easy matter to check that $C(2x^{-i}t\omega)=\sum c_{pn+i}x^n\d x$; if the said congruence would hold we get $C(2x^{-i}t\omega)=c_i\d x/(1-x)$. However the right side has a pole at points where $x=1$, whereas the left side is regular there, contradicting what has been just proved.

Here are other applications:

\begin{cor}\label{C.deg} For a differential form $\xi\in\Omega_L$, we have \begin{equation*}
\sum_{v_Q(C(\xi))<0}v_Q(C(\xi))\ge \sum_{v_Q(\xi)<0}v_Q(\xi).
\end{equation*}

\end{cor}

\begin{proof} Let $S$  be the set of poles of $\xi$. For $Q\in S$, by the preceding proposition we have $v_Q(C(\xi))\ge {v_Q(\xi)\over p}-{p-1\over p}\ge v_Q(\xi)$, because $v_Q(\xi)+1\le 0$. Summing over $Q\in S$ we obtain the assertion.
\end{proof}

We can use this corollary to give another proof of Proposition \ref{P.estimate}, actually in more precise and general form:

\begin{prop}\label{P.estimate2} Suppose that a form $\xi=w\d x$, where $w\in k_0[[x]]\cap L$, is not exact in $L$. Then there is a nonvanishing coefficient of $x^{mp-1}$ in the expansion for $w$ such that $m\le  -\sum_{v_Q(\xi)<0}v_Q(\xi)+2g-1+\deg x+N_x$, where $g$ is the genus of $L$ and $N_x$ is the number of distinct poles of $x$. 
\end{prop}

\begin{proof} We write $w=w_0^p+w_1^px+\ldots +w_{p-1}^px^{p-1}$ with $w_i\in L$, so $C(w\d x)=w_{p-1}\d x$. If the coefficient of $x^{mp-1}$ in the expansion for  $w$ vanish for $m\le l$, then the coefficients of $x^m$ in $w_{p-1}$ vanish for $m<l$, so $w$ has a zero of order $\ge l$ at the place corresponding to the expansion.  On the other hand,   in view of (C1), $w_{p-1}\neq 0$, because the form is not exact, whence $\deg w_{p-1}\ge l$.

 Also, by the last corollary, the number of poles (with multiplicity) of $w_{p-1}$ does not exceed the number of poles of $w\d x$ plus the number of zeros of $\d x$. Now, the Hurwitz formula yields $2g-2=\sum_Qv_Q(\d x)=\sum_{zeros}v_Q(\d x)+\sum_{poles}v_Q(\d x)$. The last term is $\ge -\deg x-N_x$, hence $\sum_{zeros}v_Q(\d x)\le 2g-2+\deg x+N_x$. Combining with the previous facts we get $l\le -\sum_{v_Q(\xi)<0}v_Q(\xi)+2g-2+\deg x+N_x$, as required.

\end{proof}

\begin{rem}\label{R.estimate2}
(i) In particular, if  the form ${\tilde R(x)\over 2(1+2x)^2}\omega$ on the right of (\ref{E.derivative}) is not exact in $\F_p[[x]]$, an easy computation using this last result shows that there is an $n\le 8\cdot p$ such that $C_n$ is not $p$-integral.

(ii) This result makes it effectively possible to decide whether a   differential form in $L$ is exact, provided only we can effectively calculate any prescribed number of coefficients in its expansion.
\end{rem}

\medskip





\subsection{The Cartier operator on elliptic curves} \label{SS.cartierE} 
Let now $E$ be an elliptic curve over a finite field, chosen for simplicity as $\F_p$. We suppose for this subsection that it is given by a Weierstrass equation $y^2=f(x)$ where $f$ is a monic cubic polynomial in $\F_p[x]$. \footnote{ Hence, the variables $x,y$ here, as well as $\omega,\eta$ below,  do not have the same meaning than in most of the rest of the paper, where we consider a special elliptic curve, introduced in \S\ref{SS.elliptic}.} 

We let $\omega=\d x/y$ be a differential form of the first kind (i.e. regular) and $\eta =x\omega$ a form of the second kind, i.e. with no residues; it has a double pole at the origin. As we have already recalled, in zero characteristic these forms are independent modulo exact forms  and generate the space of differential forms of the second kind modulo exact ones (see \cite{[L2]}).  None of these assertions remains true  in positive characteristic. For instance, the forms $x^{mp-1}\d x$ are linearly independent over $\F_p$ modulo exact.  But their  only  pole is at the origin of $E$; hence, there is a vector space of  codimension $1$ in the infinite-dimensional space that they span, which is made up of forms with no residues. As to the other assertion, we shall see in the next proposition that $\omega,\eta$ are in fact always linearly dependent over $\F_p$ modulo exact.   Note that exact forms are precisely those in the kernel of the Cartier operator, as recalled in (C1) above. Hence  it is relevant to inspect how the Cartier operator acts on $\omega,\eta$.   

Before stating  the said proposition, recall that we have a Frobenius map $\pi$ on $E$ and its transpose $\hat \pi$ (see \cite{[L]} or \cite{[Sil]}). The composition $\pi\cdot\hat\pi$ equals the multiplication-by-$p$ map, whereas the sum $\pi+\hat\pi$ is multiplication by an integer $A$, called the "trace of Frobenius" and we have $\# E(\F_p)=p+1-A$. 

\begin{prop} \label{P.alphabeta}There are $\alpha, \beta\in \F_p$ such that $C(\omega)=\alpha\omega$, $C(\eta)=\beta\omega$. Also, $\alpha$ is the   reduction modulo $p$ of the trace of Frobenius, so that  $\# E(\F_p)\equiv 1-\alpha\pmod p$. In particular, $\omega,\eta$ are linearly dependent modulo exact forms and satisfy the relation $\beta\omega\equiv\alpha\eta\pmod{{\rm exact \  forms}}$.

\end{prop}

\begin{proof} The  statements about $\omega,\alpha$ are   proved in \cite{[L]}, Appendix 2.  As to $\eta$, we observe that in view of Proposition \ref{P.v_Q}, $C(\eta)$ is regular outside the origin and has at most a simple pole at the origin. However the sum of the residues of a differential form is zero (see \cite{[S3]}, Ch. II). 
Hence $C(\eta)$ is regular everywhere and must therefore be a multiple $\beta\omega$ of $\omega$. \footnote{ Having an explicit equation, this may be also deduced directly from  $yC(\eta)=C(y^p\eta)=C(xf(x)^{p-1\over 2}\d x)=constant\cdot \d x$.} Since both are defined over $\F_p$, we have $\beta\in \F_p$, which concludes  the proof, except for the last assertion. For this, note that $C(\beta\omega-\alpha\eta)=0$, so $\beta\omega-\alpha\eta$ is exact by (C1). If $\alpha,\beta$ do not both vanish, this implies the said dependence; if they are both zero, then $C(\omega)=C(\eta)=0$ so both forms are in fact exact.

\end{proof}

We shall now see that indeed $\alpha,\beta$ cannot both vanish. 

Note that, while $\alpha$ is uniquely determined if we prescribe that $\omega$ is defined over $\F_p$ (but it may change if we multiply $\omega$ by a constant not in $\F_p$), the quantity $\beta$ depends on the Weierstrass equation;  however changing this equation merely changes $\eta$ with a linear combination of $\eta$ and $\omega$ and hence the nonvanishing is not affected by such a change.

\begin{prop}\label{P.alphabeta2} The above defined constants $\alpha,\beta$  cannot simultaneously vanish.

\end{prop}

\begin{proof}  We first consider certain identities. Let $\lambda$ be a variable.
 We put $m:=(p-1)/2$ and we write
\begin{equation*}
(x-1)^m(x-\lambda)^m=\sum_{i=0}^{p-1}H_i(\lambda)x^i,
\end{equation*}
for suitable polynomials $H_i$, defined to be zero for $i<0$ and $i\ge p$. This yields
\begin{equation*}
(x-1)^m(x-\lambda)^{m+1}=\sum_{i=0}^{p}(H_{i-1}(\lambda)-\lambda H_i(\lambda))x^i.
\end{equation*}
Differentiating this equation with respect to $\lambda$ (which we indicate with a dash) we get
\begin{equation*}
-(m+1)(x-1)^m(x-\lambda)^{m}=\sum_{i=0}^{p}(H_{i-1}(\lambda)-\lambda H_i(\lambda))'x^i.
\end{equation*}
Putting $K_i(\lambda):=H_{i-1}(\lambda)-\lambda H_i(\lambda)$, by comparison with the previous equation we obtain
\begin{equation*}
K_i(\lambda)'=(H_{i-1}(\lambda)-\lambda H_i(\lambda))'=-(m+1)H_i(\lambda),\qquad i=0,1,\ldots .
\end{equation*}

On expanding the binomials $(x-1)^m$ and $(x-\lambda)^{m+1}$ we can calculate explicitly
\begin{equation*}
(-1)^{m+1}K_m(\lambda)=\lambda^{m+1}+{m\choose 1}{m+1\choose 1}\lambda^m+{m\choose 2}{m+1\choose 2}\lambda^{m-1}+\ldots + {m\choose m}{m+1\choose m}\lambda.
\end{equation*}
Hence, setting $z:=\lambda^{-1}$, we also have
\begin{equation*}
(-1)^{m+1}{K_m(\lambda)\over \lambda^{m+1}}=1+{m\choose 1}{m+1\choose 1}z+{m\choose 2}{m+1\choose 2}z^2+\ldots + {m\choose m}{m+1\choose m}z^m.
\end{equation*}
In the (usual) notation of \cite{[WW]} (see 14.1, p. 281) the right hand side  is the hypergeometric function (now a polynomial) $F(z):=F(-m,-m-1,1,z)$. As in \cite{[WW]}, 14.2, p. 283, it satisfies the differential equation
\begin{equation*}
z(1-z)F''(z)+(1+2mz)F'(z)-m(m+1)F(z)=0.
\end{equation*}

Observe that, over a field of characteristic $0$ or $p$,  this implies that $F(z)$ has no multiple zero $z_0\neq 0,1$. In fact, let $\mu>1$ be the multiplicity of such a $z_0$ as a zero of $F$, so $\mu <p$. Then $z(1-z)F''(z)$ would have $z_0$ as a zero of multiplicity $\mu -2$, which contradicts the differential equation, since the other two terms would have multiplicity $\ge \mu -1$ at $z_0$. 

\smallskip

Now let us apply all of this to our issue. For the sought result it does not matter neither to change the Weierstrass equation (because $\omega$ gets multiplied by a constant  and $\eta$ gets transformed into a linear combination of $\eta$ and $\omega$), nor to extend the ground field, so we may assume that the equation for $E$ is in Legendre form $y^2=x(x-1)(x-\lambda_0)$, for some $\lambda_0\neq 0,1$ in a finite field $k_0$. We have
\begin{equation*}
yC(\omega)=C(y^{p-1}\d x)=C(x^m(x-1)^m(x-\lambda_0)^m\d x).
\end{equation*}
The polynomial $x^m(x-1)^m(x-\lambda_0)^m$ has degree $3(p-1)/2<2p-2$, so by definition the right side is precisely $H_m(\lambda_0)^{1/p}\d x$. Hence if $\alpha=0$ we have $H_m(\lambda_0)=0$.

Similarly, we obtain $yC(\eta)=H_{m-1}(\lambda_0)^{1/p}\d x$,  hence if also $\beta=0$ we have $H_{m-1}(\lambda_0)=0$.

But then, by the above identities, we have $K_m(\lambda_0)=K_m'(\lambda_0)=0$. However this shows that $F(z)$ would have $z_0:=1/\lambda_0\neq 0,1$ as a multiple zero, a contradiction which concludes the proof.

\end{proof}


\medskip

\begin{rem} (i) The argument  to show that there are no multiple zeros goes back at least to Igusa and is familiar in the theory of the Hasse invariant (see \cite{[Sil]}). 

(ii) The proof also shows explicit formulae for the said constants $\alpha,\beta$, which also allows to compute them from a Legendre equation for $E$.  
Another way to  calculate them comes on  looking at the relevant coefficients in an expansion of $\omega,\eta$ at some place.

(iii) Note that $\alpha=0$ precisely when $E$ is supersingular (see \cite{[L]} or \cite{[Sil]}), which also amounts to the fact that there are no points of order $p$ (over an algebraic closure of the ground field). 
We do not know about a possible related interpretation of the condition $\beta=0$; as $\alpha=0$ is related to "periods"of $E$ (see \cite{[Sil]}), this seems related to periods in a non-split  extension of $E$ by the additive group $\G_a$.

\end{rem}

\medskip

\subsection{Conclusion of the proof of Theorem \ref{T.modp}}\label{SS.conclusion}
As in the proof of Proposition \ref{P.residues}, for every solution of the recurrence \eqref{E.recurrence} over $\F_p$, we obtain an exact   differential form $\xi:={\tilde R(x)\over 2(1+2x)^2}\omega$ (as on the right of (\ref{E.derivative})) in $\F_p[[x]]$,  
where  $\tilde R(x):=(4\bar C_2-8\bar C_1)+(8\bar C_3+\bar C_1)x+ (12\bar C_4+8\bar C_3+2\bar C_2+3\bar C_1)x^2$.

We have already shown that, since it has no residues,  we have $\tilde R' (-1/2)=0$, which gave the condition $6\bar C_4+\bar C_2+\bar C_1=0$. Setting for this proof $t:=2x+1$, this yields 
\begin{equation*}
\xi= \left(K_1+{K_2\over t^2}\right)\omega,
\end{equation*}
where $K_1:= {1\over 4}(-31\bar C_1+18\bar C_2-8\bar C_3+12\bar C_4)$,  $K_2:={1\over 4}(12\bar C_4+8\bar C_3+2\bar C_2+3\bar C_1)$.

By (C1), we have $C(\xi)=0$. 

We now refer to our elliptic curve $E:y^2=4+(x+x^2)^2$, considered in \S \ref{SS.elliptic}, whose notation we maintain. So in particular we let $\eta=(x^2+x)\omega$ be a form of the second kind. In view of propositions \ref{P.alphabeta} and \ref{P.alphabeta2}, adapting to  the present notation one easily finds that  $C(\omega)=\alpha'\omega$, $C(\eta)=\beta'\omega$, for certain $\alpha', \beta'\in \F_p$, not both $0$.

Now, one may easily check the formula
$
\d (y/t)=(\eta/2)+(\omega/8)-(65\omega/8t^2).
$
Then, using $C(\d (y/t))=C(\xi)=0$, we derive the existence of $\gamma,\delta\in \F_p$, not both zero, and such that $\gamma (-31\bar C_1+18\bar C_2-8\bar C_3+12\bar C_4)+\delta(12\bar C_4+8\bar C_3+2\bar C_2+3\bar C_1)=0$. This provides a linear form vanishing on $V_p$. 

Observe that  this linear form is not identically zero and  that the vectors $(-31,18,-8,12)$, $(3,2,8,12)$ and $(1,1,0,6)$ are linearly independent  over $\F_p$, for $p\neq 2,3,5,13$.  Therefore the said linear form is linearly independent of $6\bar C_4+\bar C_2+\bar C_1$.  But  this last linear form vanishes as well on $V_p$, as we have already proved, and hence $\dim V_p\le 2$. Since the opposite inequality was already shown true, this concludes the proof of the theorem.


\medskip

\subsection{Some remarks} \label{SS.alphabeta}

(i)  It is not generally clear to us how, for a fixed elliptic curve $E$ over $\Q$, the distribution of $\alpha=\alpha_p$ and $\beta=\beta_p$  is governed as $p$ varies and we reduce $E$ modulo $p$. (This fits into the type of questions  considered in  Serre's most recent book \cite{[S4]}.)  

After part of this note was written I asked Katz about possible interpretations of the pair $(\alpha_p,\beta_p)$. In the note \cite{K2} he made several remarks, explaining completely the CM case, and pointing out numerical data for the general case.

\medskip


For $\alpha_p$, a celebrated deep result of Elkies \cite{[E]} affirms that $\alpha_p=0$ for infinitely many primes. In the case of Complex Multiplication, it may be proved that this holds actually for a set of primes of positive density (see below for an explicit easy example);   in the other cases such a density is zero (see  \cite{[S2]}, Thm. 3.6.3 for this and much more striking results on the  {\it diophantine} distribution of the values of the trace of Frobenius), but still no quantitative  information on the distribution (for instance as in conjectures of Lang-Trotter) seems to be known beyond Elkies'  infinitude; and for number fields other than $\Q$ not even the infinitude is generally known. 



\medskip

(ii)  In particular, given integers $r,s$ not both zero, we do not  generally know about the distribution of the primes $p$ for which $r\alpha_p+s\beta_p=0$ in $\F_p$. If there is no Complex Multiplication, we expect them to be somewhat {\it sparse}, but cannot prove this. (See (iii) below for a case of CM, and see Katz's paper \cite{K2} for the general CM case and experimental results.) However, using Andr\'e's theorem already invoked for Theorem \ref{T.converse}(ii), we can easily prove that 


\smallskip

{\it For $r,s$ not both zero, the set of primes such that $r\alpha_p+s\beta_p=0$ in $\F_p$ has not density $1$.}

\smallskip

Indeed, otherwise the form $r\omega+s\eta$ would be exact  on  the reduction of $E$ modulo $p$, for a set of primes of density $1$.  But then by Andr\'e's theorem \cite{[A2]}, Prop. 6.2.1, already recalled and used in \S\ref{S.andre}, this form would be exact as a form on $E$, which is not true (an easily seen fact already pointed out in \S\ref{SS.elliptic}).

\smallskip

This also implicitly shows a connection between this question and the above mentioned  conjecture of Grothendieck.  

The said result of Elkies also has an implication in this sense, proving that   if $rs\neq 0$ then such congruence cannot hold for all but finitely many primes: indeed, if $p$ is one of the infinitely many supersingular primes provided by Elkies' theorem, we have $\alpha_p=0$, hence $\beta_p\neq 0$ by Proposition \ref{P.alphabeta2}, so $s\beta_p$ is not zero modulo $p$ if $p$ is large enough.  

In particular, this fact combined with the above proof shows that 

\smallskip

{\it For every $2$-dimensional subspace $V$ of $\Q^4$, there exist infinitely many primes $p$ such that the space $V_p$ is not the reduction modulo $p$ of $V$}.  \footnote{ One may define reduction of a vector subspace of $\Q^m$ for instance through Grassmannian coordinates, or else through a given basis, the reduction  being anyway well defined for large enough primes.} 

\smallskip

In turn, from this  one can derive a weak form of Theorem \ref{T.converse}(ii) without the above  appeal to Andr\'e's theorem. Namely, we may prove that:

\smallskip

 {\it If the vector $(C_1,C_2, C_3,C_4)$ is not proportional to $(1,2,-1/8,-1/2)$, then there are infinitely many primes $p$ such that the sequence $(C_n)$ is not $p$-integral.}  
 
 \smallskip
  
 In fact, suppose that  $(C_n)$ is $p$-integral for all but finitely many primes $p$. Forgetting these last primes, we have that the reduction modulo $p$ of $(C_1,C_2, C_3,C_4)$ lies in $V_p$. 
 If $(C_1,C_2, C_3,C_4)$ is not proportional to $(1,2,-1/8,-1/2)$, these vectors span a $2$-dimensional space $V$ over $\Q$, whose reduction would be always contained in $V_p$. However Theorem \ref{T.modp} states that $V_p$ has dimension $2$, and thus $V_p$ would be the reduction of $V$ modulo $p$ for all but finitely many primes, a contradiction.
 
 \medskip

(iii) Here is a simple example of what may happen in the CM case. Consider the curve defined by $y^2=x^3+1$, which has CM by $\Z[\theta]$ where $\theta$ is a primitive cubic root of $1$. It has good reduction at primes $p\neq 3$.

By  observing that for $p\equiv 2\pmod 3$ the number of its points over $\F_p$ is $p+1$ we see that $\alpha_p=0$ for these primes; or else, we may  argue as in Proposition \ref{P.alphabeta2}: we have $yC(\omega)=C((x^3+1)^{p-1\over 2}\d x)$; in turn, this is computed just by picking the coefficient of $x^{p-1}$ in  $(x^3+1)^{p-1\over 2}$, since this polynomial has degree $<2(p-1)$. But this coefficient, and hence $\alpha_p$,  is clearly zero for these primes.  As to $\beta_p$, it is nonzero for this primes in virtue of Proposition \ref{P.alphabeta2}. 

On the other hand, suppose $p\equiv 1\pmod 3$. Now the same computation shows that $\alpha_p\neq 0$, and actually $\alpha_p={{p-1\over 2}\choose {p-1\over 3}}$. As to $\beta_p$, a similar argument  shows that it is the coefficient of $x^{p-1}$ in 
$x(x^3+1)^{p-1\over 2}$. However for these primes such coefficient plainly vanishes. 

In conclusion, for this example we always have $\alpha_p\beta_p=0$, but  $\alpha_p, \beta_p$ never both vanish. 



\medskip
 
 \section{Some remarks on further issues}
 
 \medskip

\begin{example}\label{Ex.elkies} 

Let $E/\Q$ be an elliptic curve, and let $\omega$ be as above a dfk.  For supersingular primes we have $\alpha_p=0\pmod p$. Using Elkies's result on the existence of infinitely many such primes, this yields an alternative proof of Polya expectation for the case of genus $1$ over $\Q$. In fact if  $(\alpha_p:\beta_p)\in \P_1(\F_p)$ is almost always orthogonal to  the reduction of a fixed $(r:s)\in\P_1(\Q)$ then this would imply $s=0$, bu  we can't have $\alpha_p=0$ always, so for those primes  the form is not exact mod $p$.

The result of Elkies also shows that there are `nontrivial' cases when a differential becomes exact mod $p$ for a set of primes of density $>0$ (actually $=1/2$) but the differential is not exact. See the Open Question below for a stronger requirement. 
\end{example}

 One could ask the following stronger form of Polya problem:
 
 \medskip
 
 {\bf Open Question}: {\it Suppose that $f(x)\d x$ is exact in $\Z_p[[x]]$ for infinitely many primes. Is it true that necessarily $f(x)\d x$ is exact ?}

 \medskip
 
 The transcendence methods of Andr\'e and Chudnovski do not  seem to yield results in this direction. 
 
 By the above we surely obtain that $f(x)\d x$ has zero residues.  This settles the issue in genus $0$, where however it suffices to have exactness over $\F_p$. 
 
 \medskip
 
 In the case of genus $1$ the information over $\F_p$ does not suffice, since $\omega$ is exact mod $p$ for supersingular primes.
 
 Over $\Z_p$ we have congruences of Atkin-Swinnerton Dyer which can give some information.  See for example \cite{LL} for this, for instance the following formula is given for the coefficients of a dfk $\omega=\d x/2y=\sum_{n\ge 1}c_nt^{n-1}\d t$ on an elliptic curve $y^2=x^3+Ax+B$, $A,B\in \Z$, where $t$ is the standard  local parameter  $-x/y$; letting $f_p$ be the trace of Frobenius, we have for $p$ of good reduction
 \begin{equation*}
c_{np^r}-f_pc_{np^{r-1}}+pc_{np^{r-2}}\equiv 0\pmod{p^r},\qquad n,r\ge 1. 
 \end{equation*}
 
 Here $c_1=1$ and $c_p=f_p$.
 
 Suppose now that $\omega$ is exact in $\Z_p[[t]]$ for some $p$. Then $c_m\equiv 0\pmod m$ in $\Z_p$. For $m=p$ we deduce that $f_p=0$, whence $pc_1\equiv 0\pmod{p^2}$, which does not hold.

 \medskip

 QUESTION: Does this give useful  information for $r\omega+s\eta$ ?

 \medskip
 
 OTHER ARGUMENT 
 
 Suppose that $\omega$ is exact in $\Z_p[[t]]$ for infinitely many primes.  Then $\ell(t):=\int \omega$ is in $\Z_p[[t]]$ and yields an isomorphism between the formal groups of $E$ and $\G_a$. This is absurd, e.g. since $[p]$ is not zero on $E\pmod p$, whereas it is $0$ on $\G_a$. 
 
 Alternatively, we can use that the radius of convergence in $\Z_p$ is $1$ for $\ell(t)$ and its inverse $\exp_E$, so we would not obtain torsion points in the disk, which is contrary to what happens.

 NOTE: Here $t$ represents the function $x/y$ in the Weierstrass coordinates for $E$, and is a local parameter at the origin.
 
 \medskip
 
 QUESTION: Does this extend somewhat to $r\omega+s\eta$ ? Maybe looking at  an extension of $E$ by $\G_a$ ?
 
 \medskip
 
 Let us look at what happens, borrowing from \cite{CMZ} the analytic formulae, see pp. 26--28 therein. Letting $G$ be a nontrivial extension of  $E$ by $\G_a$, there is an exponential map $\exp_G:\C^2\to G$, where the latter is embedded in $\P_4$, given by
 \begin{equation*}
 \exp_G(z,w):=(1:\wp(z):\wp'(z):w+\zeta(z):(w+\zeta(z))\wp'(z)+2\wp^2(z)).
 \end{equation*}
 where $\zeta(z)$ is the Weierstrass  zeta-function, such that $\zeta'(z)=-\wp(z)$ and $\zeta(z)=z^{-1}+O(z)$ near $z=0$.
 
 To have the behaviour at the origin it is convenient to divide by $\wp'(z)$, so letting $(\xi_0:\xi_1:1:\xi_3:\xi_4)$ denote the  coordinates in the open affine $\xi_2=1$, , we have $t=\xi_1$,  $w=\xi_4-\zeta(z)-2\wp^2(z)/\wp'(z)$.
 
 A corresponding  embedding in $\P_4$ appears in the paper \cite{CMZ}, (3.6), p. 21, and, in the same affine open $\xi_2=1$,  is expressed by
 \begin{equation*}
 \xi_0=4\xi_1^3-g_2\xi_0^2\xi_1-g_3\xi_0^3,\qquad \xi_0\xi_4-\xi_3=2\xi_1^2,
 \end{equation*}
 where  in homogeneous coordinates $(\xi_0:\xi_1:...)$ the line $\xi_0=\xi_1=\xi_2=0$ has to be omitted, and where the first equation defines the affine part of $E$ in the set $\xi_2=1$ (which amounts to removing from $E$ the points of exact order $2$.  The subgroup $\G_a$ is identified with the line described by $\xi_3,\xi_4$, subject to the equation in the right-hand side (which depends on the point on the elliptic curve $E$ on the left).  
 
 Again in this affine open set $\xi_2=1$ we have a section $E\to G$ given by $(\xi_0:\xi_1:1)\mapsto (\xi_0:\xi_1:1:-2\xi_1^2:0)$.
 
 \medskip
 
 This presentation also defines a formal group law in terms of the parameters $t=\xi_1$ and $\mu=\xi_4$ (so $\xi_0$ will be a power series in $t$ and $\xi_3$ will be defined by the right-hand side equation in terms of $t,\mu$). The group $\G_a$ now is identified with the points $(0:0:1:0:\mu)$ and the addition with such an element  is represented by $\xi_3\to \xi_3+\mu\xi_0, \xi_4\to \xi_4+\mu$.
 
 \medskip
 
 We have the `logarithms' $z=\ell(t)=\int \omega$ and $\zeta(z)=\ell^*(t):=\int \eta$, where $\omega, \eta$ are as above the standard $1$-forms. 
 
 Let us set $x(t),y(t)$ for the Weierstrass coordinates in terms of $t=x/y$.  We may view $t=t(Q)$ as a map on $G$ defined by composing with the projection to $E$. We find that a local inverse to $\exp_G$ near $t=0$  is given by the map 
 \begin{equation*}
 Q\in G\to \C^2:\qquad L(Q)=(\ell(t(Q)), \xi_4(Q)-\ell^*(t(Q))-2x(t(Q))t(Q).
 \end{equation*}
 This is a local homomorphism.
 
 We may now compose with the map $(u,v)\mapsto ru-sv$, from $\C^2$ to $\C$, which is a homomorphism, to obtain a local homomorphism
 \begin{equation*}
 \psi(Q)=r\ell(t)+s\ell^*(t)+2s\cdot x(t)t-s\cdot \xi_4(Q)=:\sigma(t)+2s\cdot x(t)t-s\cdot \xi_4(Q),
\end{equation*}
where $t=t(Q)$ and where we have denoted $\sigma(t):= r\ell(t)+s\ell^*(t)$.

 Here $r,s$ are integers as above, and we suppose that any of them is divisible by $p$ only if it vanishes (but not both vanish, by assumption). By our assumption the series $\sigma(t)=r\ell(t)+s\ell^*(t)$ has $p$-integer coefficients, so $\psi$ induces    a (nonzero) homomorphism of formal groups defined over $\Z_p$ (that is, from the formal group associated to $G$ towards the formal group associated to $\G_a$).  For $p$ of good reduction the series $x(t)$ is over $\Z_p$ as well. 
 
 \medskip
 
 \noindent{\bf Remark}. An extension $G$ of $E$ by $\G_a$ over $\F_p$ is isogenous to a split extension, since the map $Q\mapsto pQ$ has kernel containing $\G_a$, so the image  is a curve which serves for the splitting (see also \cite{V} for more precision as to this and several other facts).  However $G$ is not itself split not even over $\F_p$, as follows from the description of $G$ as a generalized Jacobian (or else as in Voloch's paper \cite{V} as a sheaf with connections). 
 
 We note that the only points of order $p$ of $G$ lie in $\G_a$.  We were unable at the moment to find a reference for this fact, which we can prove as follows.  Let $Q\in G-\G_a$ have exact order $p$, so the projection $R=\pi(Q)\in E$ has exact order $p$. This implies that the divisor $p((R)-(-R))$  is principal, say $=\div(F)$ with $F$ a rational function on $E$.  The differential $\d F$ has order $\ge p$ at $R$ and $\le p$ at $-R$, so either vanishes or has divisor $=\div(F)$. The first option is impossible, since $F$ cannot be a $p$-th power. Then $\d F= cF\cdot \omega$, for a nonzero constant $c$. But then the interpretation of $G$ as a generalized Jacobian (see \cite{[S3]} and \cite{CMZ}) proves that $pQ$ is not $0$ on $G$. Since modifying $Q$ by a point of $\G_a$ does not change $R$, this proves the conclusion. 

(Probably another proof follows on noting that  $G$ modulo $p$ is not split, so not isomorphic to $E\times \G_a$, though isogenous to it. Then the image of $[p]$ on $G$ must intersect $\G_a$, proving the contention.) 

\medskip

{\bf A possible argument}. Let us call $G_p$ the reduction of $G$ modulo $p$. As observed above, this is non-split but isoegnous to a split extension.

Indeed,  the multiplication map $Q\mapsto [p]Q$ on $G$ annihilates $\G_a$, whence the image is a(n affine) curve, whose closure in $G$ we denote $H$; in fact the image is closed because it is a subgroup, so $G_p=H$.

The (formal) kernel $K$ of $\psi$ is given by the equation $\xi_4(Q)=s^{-1}\sigma(t)+2tx(t)\in \Z_p[[t]]$ (in fact the poles cancel).  We observe that $K$ is a formal subgroup of the formal group corresponding to $G$, which we denote by $\widehat G$.

We can reduce modulo $p$ the series $\sigma(t)$ and $x(t)$.  The map $[p]$  sends $\widehat G_p$ to the reduction $\widehat K_p$ of $\widehat K$, since $\psi$ goes to $\G_a$, which is annihilated by $[p]$. 

The  above mentioned section  $(\xi_0:\xi_1:1)\mapsto (\xi_0:\xi_1:1:-2\xi_1^2:0)$ yields a formal section $\widehat E\to \widehat G$, which can be reduced modulo $p$ and multiplied by $p$, hence we obtain a formal map $\gamma: \widehat E_p\to \widehat K_p$ which is nontrivial, since $[p]$ is nontrivial on $E_p$.  Therefore $\widehat K_p$ is the formal image of $\widehat E_p$ and we also deduce that $K_p$ is an algebraic curve in $G_p$.  Since it is a subgroup it is isomorphic necessarily to $E_p$ (under the projection map). This yields a contradiction with the fact that $G_p$ is not split. 

\medskip

NOTE: (i) Observe also that $E_p$ has points of order $p$ in the ordinary case, hence we obtain a contradiction with a deduction above. On the other hand, in the supersingular case we have $\alpha_p=0$ hence $\beta_p$ cannot be zero, and hence $s=0$ and we go back to the case of elliptic curves treated at the beginning of this section. 

(ii) Probably a different way to argue is in characteristic zero, on using a point of order $p$ congruent to the origin, and evaluating the various series at that point. 

\medskip

Thus we have proved an improved form of the Polya assumption the genus 1 case:

\begin{thm}
Let $f(x)\in \Z[[x]]$ be a series representing a rational  function on an elliptic curve, and suppose that for infinitely many primes $p$ the integral of $f(x)\d x$ has $p$-integer coefficients. Then the integral is also an algebraic function.
\end{thm}

\medskip

 \section{Further miscellaneous remarks}
 
 {\bf 1}. The operation on series $s(x)\mapsto s^\phi(x):=p^{-1}(\sum_{\theta^p=1}s(\theta x))$ sends $\sum_{n=0}^\infty a_nx^n$ to $\sum_{p|n}a_nx^n$ and hence preserves the series with integer coefficients. It is plainly relevant in some of the above considerations, since for instance it may be easily used to represent the Cartier operator in zero characteristic. 
 
 However, even if e.g. $s(x)$ is integral over $\Z_p[x]$, the same does not necessarily hold for $s^\phi$. For an example take $p=3$ and $s(x)=(1+x)^{1/2}=\sum {1/2\choose n}x^n$, of degree $2$ over $\Q_3(x)$.  The series $s^\phi$ has degree $8$ over $\Q_3(x)$ (for instance since the branch points of $s(x),s(\zeta x), s(\zeta^2 x)$, where $\zeta^2+\zeta+1=0$, do not match). 
 
 So the conjugates of $s^\phi(x)$ over $\Q_3(x)$ are the $(\pm s(x)\pm s(\zeta x)\pm s(\zeta^2x))/3$, for all possible choices of the signs. But not all the corresponding series expansions have  $3$-integer coefficients.  This leads to a contradiction with integrality over $\Z_3[x]$.

 
 \smallskip
 
 Then it seems better to use instead the $(p-1)$-th divided derivative, which as a series yields $\sum_{n=0}^\infty {n\choose p-1}a_nx^{n-p+1}\equiv \sum_{n\equiv -1}a_nx^{n-p+1}\pmod p$. 
 
 This derivative may be convenient because (in zero characteristic) it continues to lie in the field $K(x,s(x))$.

 \medskip

 \bigskip

\bigskip
\bigskip

\bigskip
\bigskip

\end{document}